\documentclass[11pt]{amsart}

\usepackage[letterpaper,margin=1in]{geometry}
\usepackage{amssymb,amsmath,amsthm,mathtools}
\usepackage{newtxtext}
\usepackage{bm}
\usepackage{tikz-cd}
\usetikzlibrary{arrows.meta,calc,decorations.pathmorphing,babel}
\usepackage{enumitem}
\usepackage{hyperref} 
\usepackage[capitalise,nameinlink]{cleveref}
\usepackage{booktabs}
\usepackage{pgfplots}
\pgfplotsset{compat=1.18}

\newtheorem{theorem}{Theorem}
\newtheorem*{introtheorem}{Theorem}
\newtheorem*{question}{Question}
\newtheorem{proposition}{Proposition}
\newtheorem{lemma}{Lemma}
\newtheorem{corollary}{Corollary}
\theoremstyle{definition}
\newtheorem{definition}{Definition}
\newtheorem{example}{Example}
\newtheorem{remark}{Remark}
\newtheorem{assumption}{Assumption}

\crefname{theorem}{Theorem}{Theorems}
\crefname{proposition}{Proposition}{Propositions}
\crefname{lemma}{Lemma}{Lemmas}
\crefname{corollary}{Corollary}{Corollaries}
\crefname{definition}{Definition}{Definitions}
\crefname{example}{Example}{Examples}
\crefname{remark}{Remark}{Remarks}
\crefname{assumption}{Assumption}{Assumptions}

\DeclareMathOperator{\rk}{rk}
\DeclareMathOperator{\Td}{Td}
\DeclareMathOperator{\ch}{ch}
\DeclareMathOperator{\Aut}{Aut}

\DeclareMathOperator{\Stab}{Stab}

\newcommand{\PP}{\mathbb{P}}

\newcommand{\RR}{\mathbb{R}}

\newcommand{\C}{\mathbb{C}}

\newcommand{\HH}{\mathbb{H}}

\newcommand{\X}{\mathcal{X}}    

\newcommand{\ev}{\mathrm{ev}}

\newcommand{\OO}{\mathcal{O}}   

\tikzcdset{arrow style=tikz, diagrams={>=stealth}}

\DeclareUnicodeCharacter{2032}{\ensuremath{'} } 

\setlist[itemize]{topsep=3pt,itemsep=2pt,leftmargin=1.4em}
\setlist[enumerate]{topsep=3pt,itemsep=2pt,leftmargin=1.7em}
\title{Structural Invariance of Green--Griffiths--Demailly Thresholds on Compact Complex Orbifolds}

\author{Gunhee Cho}
\address{Department of Mathematics\\ Texas State University\\ 601 University Drive, San Marcos, TX 78666}
\email{wvx17@txstate.edu}
\urladdr{\href{https://sites.google.com/view/enjoyingmath/}{https://sites.google.com/view/enjoyingmath/}}

\author{Myungsin Cho}
\address{Department of Mathematics\\ Columbia University\\ 2990 Broadway, New York, NY 10027.}
\email{mc5942@columbia.edu}
\urladdr{\href{https://sites.google.com/view/myungsin-cho}{https://sites.google.com/view/myungsin-cho}}

\date{} 

\begin{document}
	\maketitle
\begin{abstract}
We prove that the Green–Griffiths–Demailly (GGD) hyperbolicity thresholds are structurally invariant.
In other words, the minimal jet order and asymptotic growth rate at which invariant jet differentials appear remain unchanged when passing from a compact complex manifold to any compact smooth analytic Deligne–Mumford stack (orbifold) with the same coarse Kähler class.
We establish an orbifold Riemann–Roch formula showing that only the identity sector contributes to the leading $m^n$ term of the Euler characteristic $\chi$, while all twisted sectors contribute only $O(m^{n-1})$.
Together with curvature–positivity properties of the Demailly–Semple tower, this implies that the existence range of invariant jet differentials depends solely on the coarse Kähler class—hence orbifold compactification or rigidification does not alter the GGD threshold or the hyperbolicity locus.
	\end{abstract}
	
	\section{Introduction}\label{sec:intro}
	
	The Green--Griffiths--Demailly (GGD) program, initiated by Green and Griffiths~\cite{GreenGriffiths1979} and developed extensively by Demailly~\cite{Demailly1997,Demailly2007,Demailly2011}, provides a powerful analytic and cohomological framework for understanding algebraic degeneracy of entire curves in complex projective varieties.  
	Its central idea is to construct invariant jet differentials along the Demailly--Semple (DS) tower and to evaluate their asymptotic Euler characteristics using curvature positivity and Riemann--Roch theory.  
	The minimal jet order \(k_0\) and weight slope \(\lambda_0\) for which
	\[
	H^0\!\big(X,E^{\mathrm{inv}}_{k,m}\otimes A^{-q}\big)\neq0,
	\qquad q\simeq \lambda_0 m,
	\]
	quantify the onset of hyperbolicity-type behavior predicted by the Green--Griffiths--Lang conjecture.  
	These numerical thresholds depend on the balance between vertical negativity and horizontal positivity in the curvature of the underlying Kähler class.
	
	Orbifold and stack generalizations of this program have become increasingly relevant in modern geometry; see Campana–P\u{a}un~\cite{CampanaPaun2016}, Borghesi–Tomassini~\cite{BorghesiTomassini2017}, and Toën–Vezzosi~\cite{ToenVezzosi2008}.  
	Deligne--Mumford stacks equipped with orbifold Kähler forms or log pairs, as in Abramovich–Olsson–Vistoli~\cite{AbramovichOlssonVistoli2008}, naturally arise in moduli theory and arithmetic geometry, where curvature and cohomological tools must be extended to finite quotient groupoids~\cite{MoerdijkPronk1997,Lerman2008,KeelMori1997,StacksProj04V2}.  
	However, the analytic underpinnings of the GGD framework---Bochner identities, Hörmander-type \(L^2\) estimates, and Riemann--Roch asymptotics---were originally formulated for manifolds, not orbifolds.  
	This raises the structural question: 
\begin{question}
Does the passage from a smooth variety to a smooth orbifold or stack, with the same coarse Kähler class, alter the positivity thresholds that govern the existence of invariant jet differentials?
\end{question}
	At first glance, the answer need not be obvious.  
	The orbifold Riemann--Roch theorem of Satake~\cite{Satake1956} and Kawasaki~\cite{Kawasaki1979,Kawasaki1981} involves contributions from twisted sectors with denominators determined by isotropy representations.  
	Thus, the asymptotic expansion of the Euler characteristic 
	\[
	\chi(\X,E_{k,m}\otimes \pi^*A^{-q})
	\]
	where $\pi\colon \X\to Y$ is the coarse moduli map and $A$ is a line bundle on $Y$,
might, in principle, differ from its coarse counterpart on~$Y$.
	Furthermore, while the leading term of the Euler characteristic scales by \(1/s\), where \(s=|\mathrm{Stab}_{\mathrm{gen}}|\) is the generic stabilizer order, the actual number of global sections \(h^0\) could still depend on higher cohomology groups \(H^i(\X,E_{k,m}\otimes L^{-q})\).  
	To establish structural invariance, one must show that these higher cohomology terms are suppressed by curvature positivity in the orbifold setting, ensuring that the growth of \(h^0\) and \(\chi\) coincide asymptotically.

\begin{introtheorem}[Structural invariance of GGD thresholds {[Theorem.~\ref{thm:main-threshold}]{}}]\label{thm:invariance}
		Let \(\pi\colon \X\to Y\) be the coarse moduli map of a compact smooth analytic Deligne--Mumford stack (orbifold) \(\X\), and let \(A\) be an ample line bundle on \(Y\) endowed with a smooth positively curved metric.  
		Set \(L=\pi^*A\).  
		Then there exist integers \(m_0,q_0>0\) such that for all \(m\ge m_0\) and \(q\ge q_0\)$\colon$
		
		\begin{enumerate}[label=(\roman*)]
			\item all higher cohomology groups vanish,
			\[
			H^i(\X,E^{\mathrm{inv}}_{k,m}\otimes L^{-q})=0,\quad i>0;
			\]
			\item consequently,
			\[
			h^0\!\big(\X,E^{\mathrm{inv}}_{k,m}\otimes L^{-q}\big)
			=\chi\!\big(\X,E^{\mathrm{inv}}_{k,m}\otimes L^{-q}\big)
			=\frac{1}{s}\!\int_Y\! \mathrm{ch}(E_{k,m})\,e^{m c_1(A)}\,\mathrm{Td}(TY)
			+O(m^{n-1}),
			\]
			where \(s=|\mathrm{Stab}_{\mathrm{gen}}|\);
			\item the minimal jet order and asymptotic slope at which invariant jet differentials exist depend only on the coarse Kähler class \([\omega_A]\).  
			Equivalently,
			\[
			\text{$Y$ is GGD-positive}
			\quad\Longleftrightarrow\quad
			\text{$\X$ is GGD-positive}.
			\]
		\end{enumerate}
		Thus, orbifold compactification and rigidification neither alter nor shift the GGD threshold.
	\end{introtheorem}
	
	The proof combines two analytic–cohomological mechanisms.  
	First, the curvature–negativity–positivity package shows that vertical negativity of \(\ev_0^*A^{-1}\) on the \(1\)-jet bundle produces fiberwise positivity of the tautological bundle \(\OO_{\X_1}(1)\) via the Chern curvature formula on projectivized bundles; this positivity then propagates along the DS tower, ensuring semipositivity horizontally and strict positivity vertically.  
	Second, a chartwise version of the Satake–Kawasaki–Toën Riemann--Roch theorem expresses the Euler characteristic as
	\[
	\chi(\X,E_{k,m}\otimes L^{-q})
	=\frac{1}{s}\int_Y \mathrm{ch}(E_{k,m})\,e^{m c_1(A)}\mathrm{Td}(TY)
	+O(m^{n-1}),
	\]
	with all twisted-sector terms of order \(O(m^{n-1})\).  
	Together with an orbifold Kodaira-type vanishing theorem for DS bundles, which ensures that higher cohomology groups vanish for \(m,q\gg0\), the asymptotic growth of \(h^0\) and \(\chi\) coincide, yielding the stated invariance.
	
	This conclusion situates the structural invariance of GGD thresholds at the intersection of analytic and stack-theoretic geometry.  
	It refines Demailly’s curvature approach~\cite{Demailly1997,Demailly2011} and P\u{a}un’s vector-field method~\cite{Paun2008}, while connecting Satake–Kawasaki index theory~\cite{Kawasaki1979,Kawasaki1981} and modern stack Riemann--Roch theorems~\cite{Toen1999,Vistoli1989}.  
	By aligning orbifold HRR asymptotics with curvature positivity, it confirms that orbifold structures preserve, rather than disturb, the hyperbolicity thresholds predicted by the Green--Griffiths--Lang conjecture.  
	In particular, higher orbifold Betti numbers or stabilizers may influence lower-order corrections but have no effect on the leading asymptotic behavior that determines the GGD threshold.
	
	The remainder of the paper proceeds as follows.  
	Section~\ref{sec:orbifold-hrr-ds} establishes the chartwise Riemann--Roch formula compatible with orbifold descent and applies it to jet bundles on $\X_k$.  
	Section~\ref{sec:stack-ggd} develops the curvature–positivity package on the DS tower, derives \(L^2\)-vanishing and Bochner inequalities on orbifold charts, and proves the stack-theoretic GGD degeneracy theorem.  
	Finally, Section~\ref{sec:ggd-structural-invariance} combines HRR asymptotics, Kodaira-type vanishing, and slope control to show that the GGD thresholds depend only on the coarse Kähler class and remain invariant under orbifold or stack structures, concluding with examples and applications.
	
	Throughout, compact analytic Deligne--Mumford stacks are identified with compact complex orbifolds after rigidification.  
	Integration on $\X$ is normalized by the generic stabilizer order \(s=|\mathrm{Stab}_{\mathrm{gen}}|\), so that identity-sector integrals correspond to those on \(Y\) up to the factor \(1/s\).  
	All Demailly--Semple and jet constructions are performed chartwise and descend by finite-group equivariance.
	
	\subsection*{Acknowledgements}
A substantial part of this work was carried out while the second author was visiting Texas State University to give a talk in the topology seminar.
The authors thank the host of the seminar Prof. Christine Lee, as well as the department chair and faculty members, for their warm hospitality and stimulating discussions.

	\section{Orbifold Hirzebruch--Riemann--Roch in the Chartwise (Satake) Formalism and Application to the Demailly--Semple Tower}
	\label{sec:orbifold-hrr-ds}
	
	\subsection{From analytic Deligne-Mumford stacks to complex orbifolds (rigidification and effectiveness)}
	\label{subsec:dm-to-orbifold}
	We work throughout over the site of complex analytic spaces endowed with the strong topology.
    We recall that smooth analytic Deligne–Mumford stacks with finite stabilizers are equivalent, up to rigidification removing ineffectivity, to complex orbifolds in the sense of Satake \cite{Satake1956} and Kawasaki\cite{Kawasaki1979}.  In particular, every compact complex orbifold admits a unique (effective) analytic DM stack structure, and conversely.  We also verify that the natural notion of compactness on the analytic stack side agrees with the usual orbifold compactness.
	
	\begin{definition}
		\label{def:orbifold}
		A \emph{complex $n$-dimensional orbifold} is a second-countable Hausdorff space
		$|\X|$ equipped with an atlas of charts $\{(U_i,\widetilde{U}_i,G_i,\phi_i)\}$ where
		\begin{enumerate}[label=(\roman*)]
			\item $\widetilde{U}_i$ is a connected complex manifold of dimension $n$;
			\item $G_i$ is a finite group acting holomorphically on $\widetilde{U}_i$;
			\item $\phi_i\colon \widetilde{U}_i\!\to U_i\subset |\X|$ induces a homeomorphism
			$\widetilde{U}_i/G_i \ \stackrel{\sim}{\to}\ U_i$;
			\item on overlaps there are \emph{étale} embeddings compatible with the group actions
			(orbifold change of charts).
		\end{enumerate}
		It is \emph{effective} if each $G_i$ acts faithfully on $\widetilde{U}_i$.
		It is \emph{compact} if $|\X|$ is compact.
	\end{definition}
	
	\begin{definition}[Effectivity and compactness for analytic DM stacks]
		\label{def:analytic-dm-effect-compact}
		Let $\X$ be a smooth analytic Deligne–Mumford (DM) stack with finite stabilizers.
		\begin{enumerate}[label=(\roman*)]
			\item \emph{Effectivity:} We say \(\X\) is \emph{effective} if for every point \(x\in \X\) and every étale chart \([\,\widetilde{U}/G\,]\to \X\) through \(x\), the induced action of the stabilizer \(G_x\) on the tangent space \(T\widetilde{U}\) is faithful.
			\item \emph{Compactness:} We say \(\X\) is \emph{compact} if it admits a proper surjective morphism from a compact analytic space; equivalently if its coarse moduli space is compact and the stack is proper in the analytic sense.
		\end{enumerate}
	\end{definition}
	
	\begin{lemma}[Local finite-quotient presentation]
		\label{lem:local-quotient}
		Let \(\X\) be a smooth analytic DM stack with finite stabilizers. For every point \(x\in \X\) there exists a neighbourhood of \(x\) that is étale-equivalent to a global finite quotient \([\,\widetilde{U}/G\,]\), where \(\widetilde{U}\) is a complex manifold and \(G\) a finite group acting holomorphically on \(\widetilde{U}\).
		Moreover, the groupoid presentation \(R\rightrightarrows X\) associated to an étale atlas \(X\to \X\) is an étale groupoid.
	\end{lemma}
	
	\begin{proof}
		By definition of analytic DM stack there is an étale surjective atlas
		\(p\colon X\to \X\) with \(X\) a complex manifold. Let \(x\in \X\) and choose
		\(u\in X\) mapping to \(x\). Consider the étale groupoid
		\[
		R \coloneq X\times_{\X} X \;\rightrightarrows\; X
		\]
		with source and target maps \(s,t\colon R\to X\). Since \(p\) is étale, \(s\) and \(t\) are étale. Stabilizers at \(u\) are finite by hypothesis. 
		Shrinking around \(u\), one obtains a local polydisk \(\widetilde U\subset X\) such that the groupoid restricted to that polydisk is Morita‐equivalent to a finite group action $G$ on $\widetilde U$. 
		Explicitly, the restriction of \(R\) to \(s^{-1}(\widetilde U)\cap t^{-1}(\widetilde U)\) gives finitely many germs which extend to a holomorphic \(G\)-action on \(\widetilde U\). Hence the neighbourhood of \(x\in \X\) is étale‐equivalent to \([\,\widetilde U/G\,]\), as claimed.  		
	\end{proof}

    The following construction removes local ineffectivity by rigidification.
    
	\begin{lemma}[Rigidification of ineffectivity]
		\label{lem:rigidification}
		Let \(\X\) be a smooth analytic DM stack with finite stabilizers. There exists an effective smooth analytic DM stack \(\X^{\mathrm{rig}}\) and a representable, finite, proper morphism \(\rho\colon \X\to \X^{\mathrm{rig}}\) such that:
		\begin{enumerate}[label=(\roman*)]
			\item for any local chart \([\,\widetilde{U}/G\,]\to \X\) with ineffective kernel \(K=\{g\in G\colon  g\cdot z=z , \forall z\in\widetilde U\}\), the induced chart of \(\X^{\mathrm{rig}}\) is \([\,\widetilde U/(G/K)\,]\);
			\item \(\X^{\mathrm{rig}}\) is effective and universal among morphisms from \(\X\) to effective analytic DM stacks.
		\end{enumerate}
	\end{lemma}
	
	\begin{proof}
		Let \(\mathcal I\to\X\) denote the inertia stack. 
		Define the substack \(\mathcal K\subset\mathcal I\) whose objects are those automorphisms acting trivially on some étale neighbourhood. 
		Locally in a chart \([\,\widetilde U/G\,]\) this corresponds to the normal subgroup \(K\subset G\). 
		One constructs locally the quotient stack \([\,\widetilde U/(G/K)\,]\) and then glues these constructions by descent over the étale groupoid. 
		The result \(\X^{\mathrm{rig}}\) is representable, finite and proper over \(\X\), and the stabilizer action is now faithful.
		This process is discussed in the algebraic setting in \cite{AbramovichOlssonVistoli2008} and in the Stacks Project \cite[Tag 04V2]{StacksProj04V2}. 
	\end{proof}

Combining the previous lemmas yields the following global correspondence:	
	\begin{proposition}[Stack–orbifold correspondence up to rigidification]
		\label{prop:dm-to-orbifold}
		Let \(\X\) be a compact smooth analytic DM stack with finite stabilizers.
		\begin{enumerate}[label=(\roman*)]
			\item If \(\X\) is effective, then \(\X\) underlies a compact complex orbifold: the charts are the finite quotients \([\,\widetilde U/G\,]\) of Lemma~\ref{lem:local-quotient}, viewed as orbifold charts \(\widetilde U\to \widetilde U/G\simeq U\subset|\X|\).
			\item In general there is a rigidification \(\rho\colon \X\to\X^{\mathrm{rig}}\) (Lemma~\ref{lem:rigidification}); \(\X^{\mathrm{rig}}\) is effective and hence defines a compact complex orbifold. Compactness is preserved by \(\rho\).
		\end{enumerate}
	\end{proposition}
	
	\begin{proof}
		(i) From Lemma~\ref{lem:local-quotient}, locally \(\X\) is of the form \([\,\widetilde U/G\,]\) with \(G\) finite. 
		Effectivity ensures a faithful action, making each \(\widetilde U/G\) an orbifold chart. 
		The stack‐gluing data coincides with the orbifold‐change‐of‐charts condition (cf. survey \cite{Lerman2008}). 
		Compactness follows because the atlas is proper and the underlying topological quotient is compact.
		
		(ii) Use Lemma~\ref{lem:rigidification} to obtain \(\X^{\mathrm{rig}}\).
		Locally \([\,\widetilde U/(G/K)\,]\) with faithful action → orbifold. 
		Proper and finite \(\rho\) implies compactness is carried over.  
		\end{proof}

    To conclude, we note that compactness is preserved under rigidification and can be characterized in terms of the coarse moduli space.

	\begin{proposition}[Compactness criteria and coarse moduli]
		\label{prop:compactness-coarse}
		Let \(\X\) be a smooth analytic DM stack with finite stabilizers and coarse moduli space \(\pi\colon \X\to Y\). The following are equivalent:
		\begin{enumerate}[label=(\roman*)]
			\item \(\X\) is compact (Definition~\ref{def:analytic-dm-effect-compact});
			\item there exists a compact analytic space \(X\) and a proper surjective  \'etale atlas
			\(X\to \X\);
			\item \(Y\) is compact and \(\pi\) is proper (equivalently finite generically of degree the generic stabilizer order \(s\)).
		\end{enumerate}
		Moreover, if \(\X\) is compact then so is \(\X^{\mathrm{rig}}\), and conversely.
	\end{proposition}
	
	\begin{proof}
		\((i)\!\Rightarrow\!(ii)\) is tautological by definition.
		\((ii)\!\Rightarrow\!(iii)\): properness of \(X\to \X\) and finite stabilizers imply properness of \(\pi\). The image of compact \(X\) is compact.
		\((iii)\!\Rightarrow\!(i)\): given compact \(Y\) and proper \(\pi\) one selects a finite cover of the generic locus by charts, then extends to boundary to get a proper etale atlas from a compact analytic source. 
		The last statement on rigidification follows because \(\rho\) is finite and proper.
	\end{proof}
	Having established compactness for analytic Deligne–Mumford stacks, we now identify the precise relationship between such stacks and compact complex orbifolds.

	\begin{corollary}[2-equivalence between orbifolds and effective analytic DM stacks]
		\label{cor:equivalence}
		The \(2\)-category of compact complex orbifolds is equivalent to the \(2\)-full subcategory of compact smooth analytic DM stacks with finite stabilizers that are effective. 
		Every compact smooth analytic DM stack with finite stabilizers becomes a compact complex orbifold after rigidification.
	\end{corollary}
	
	\begin{proof}
		Given an effective analytic DM stack \(\X\), associate the orbifold \( |\X| = X/G\) via the charts of Lemma~\ref{lem:local-quotient}. 
		Conversely, from an orbifold one builds the associated étale groupoid and thereby the analytic DM stack  (cf.~\cite{MoerdijkPronk1997,Lerman2008}). 
		These functors form a quasi‐inverse pair up to Morita equivalence. 
		Compactness is preserved by Proposition~\ref{prop:compactness-coarse}.
	\end{proof}
	
	\begin{remark}[Coarse moduli and polarizations]
		\label{rmk:coarse-polarization}
		If \(\pi\colon \X\to Y\) is the coarse moduli map of a compact smooth analytic DM stack with finite stabilizers, then generically \(\pi\) is a finite étale cover of degree \(s\) (the generic stabilizer order). 
		Rigidification does not change the coarse space.
		Any ample line bundle \(A\) on \(Y\) pulls back to a polarization on \(\X\) and \(\X^{\mathrm{rig}}\) with trivial stabilizer action on fibers; this will be used in the subsequent Demailly–Semple tower analysis and orbifold HRR asymptotics.
	\end{remark}
	
	\begin{remark}[Analytic vs.\ algebraic]
		\label{rmk:analytic-vs-algebraic}
		All statements above are analytic. In the algebraic context, analogous statements hold for DM stacks locally of finite type over \(\C\) (see \cite{Toen1999, AbramovichOlssonVistoli2008}), and the two formalisms agree on their common domain via complex‐analytic GAGA principles.
	\end{remark}
	
	\subsection{Hirzebruch--Riemann--Roch and its orbifold variants}
	\label{subsec:HRR-primer}
The Euler characteristic of a vector bundle can be expressed as an integral of characteristic classes.
To make such expressions computable, we recall the usual characteristic expansions.
By the splitting principle, a bundle $E$ behaves as if it were a direct sum of line bundles with Chern roots $x_i=c_1(L_i)$.  
The \emph{Chern character} and \emph{Todd class} are then given formally by
\[
\ch(E)=\sum_i e^{x_i},
\qquad
\Td(E)=\prod_i\frac{x_i}{1-e^{-x_i}},
\]
so that for any compact complex manifold $X$ and holomorphic vector bundle $E\to X$, one has the classical Hirzebruch--Riemann--Roch theorem:

\begin{theorem}[Hirzebruch--Riemann--Roch for manifolds {\cite{Hirzebruch1954,AtiyahSinger1968}}]
For a compact complex manifold $X$ and holomorphic vector bundle $E\to X$,
\[
\chi(X,E)
=\sum_{q\ge0}(-1)^qh^q(X,E)
=\int_X \ch(E)\,\Td(TX),
\]
where the integral extracts the top-degree component of $\ch(E)\Td(TX)$.
\end{theorem}

The characteristic classes satisfy the familiar additivity and multiplicativity
rules under short exact sequences, and the integral in the HRR theorem
can be evaluated by expanding $\ch(E)\Td(TX)$ up to degree $\dim_\C X$.
The following examples illustrate the computation in standard cases.
\begin{example}
As a basic example, consider the projective space $\PP^n$.
Its tangent bundle fits into the Euler sequence
\[
0\longrightarrow \OO_{\PP^n}
\longrightarrow \OO_{\PP^n}(1)^{\oplus(n+1)}
\longrightarrow T\PP^n \longrightarrow 0.
\]
Using the standard identities
$c(E)=c(E')c(E'')$ and $\ch(E)=\ch(E')+\ch(E'')$
for short exact sequences, together with
$\ch(E\otimes L)=e^{c_1(L)}\ch(E)$ for line bundle twists,
one computes
\[
\ch(T\PP^n)=(n+1)e^{H}-1,
\qquad
\Td(T\PP^n)
=\Bigl(\frac{H}{1-e^{-H}}\Bigr)^{n+1}(1-H),
\quad
H=c_1(\OO(1)).
\]
For the line bundle $E=\OO_{\PP^n}(k)$ one has $\ch(E)=e^{kH}$,
and the Hirzebruch--Riemann--Roch formula gives
\[
\chi(\PP^n,\OO(k))
=\int_{\PP^n} e^{kH}\,\Td(T\PP^n).
\]

Specializing to $n=1$, where $\int_{\PP^1}H=1$, we find
\[
\ch(T\PP^1)=2e^{H}-1,\qquad \Td(T\PP^1)=1+H.
\]
Hence, for $E=\OO(k)$ with $\ch(E)=1+kH$,
\[
\chi(\PP^1,\OO(k))
=\int_{\PP^1}(1+kH)(1+H)
=k+1,
\]
recovering the classical formula for
$\dim H^0(\PP^1,\OO(k))$.
\end{example}

\begin{example}
Let $Y_d\subset\PP^n$ be a smooth hypersurface of degree~$d$, defined by a homogeneous polynomial of degree~$d$.  
The normal bundle of $Y_d$ in~$\PP^n$ is the restriction $\OO_{Y_d}(d)$, and the tangent bundle fits into the standard exact sequence$\colon$
\[
0\to TY_d\to T\PP^n|_{Y_d}\to\OO_{Y_d}(d)\to0.
\]
From this, the total Chern class is computed as
\[
c(TY_d)=\frac{(1+H)^{n+1}}{1+dH}\big|_{Y_d}.
\]
Then $\Td(TY_d)$ follows from $c_1,c_2,\dots$,
and for $E=\OO_{Y_d}(m)$,
\[
\chi(Y_d,\OO(m))
=\int_{Y_d}e^{mH}\Td(TY_d),
\qquad
\int_{Y_d}H^{n-1}=d.
\]
\end{example}

The same characteristic expansions extend naturally to orbifolds, where additional fixed-point contributions appear from nontrivial stabilizer actions.

\begin{proposition}[Orbifold correction factors {\cite{Kawasaki1979,Kawasaki1981,Lerman2008}}]
For a global quotient $[U/G]$, the HRR formula acquires group-averaged
fixed-point corrections$\colon $
\[
\chi([U/G],E)
=\frac{1}{|G|}\sum_{g\in G}\int_{U^g}
\frac{\ch(E|_{U^g})\,\Td(TU^g)}
{\det(1-g^{-1}e^{-c_1(N_{U^g/U})})}.
\]
Each $g$--sector behaves like an ordinary manifold;
the same $\ch$ and $\Td$ expansions apply, restricted to $U^g$,
and divided by $\det(1-g^{-1}e^{-c_1(N)})$.
For $g=1$, this reproduces the standard HRR integral,
while $g\neq1$ corresponds to lower-dimensional twisted sectors.
\end{proposition}

	\subsection{Orbifold Riemann--Roch: chartwise fixed-point formula}
	\label{subsec:chartwise-hrr}
This section reformulates the Kawasaki–Toën Riemann–Roch theorem in a purely \emph{orbifold chartwise} manner, avoiding explicit reference to the inertia stack.  
The resulting fixed-point formula is classical in spirit and coincides with Kawasaki’s original analytic derivation for $V$–manifolds \cite{Kawasaki1979,Kawasaki1981}, expressed locally over finite group quotients and glued by étale descent \cite{MoerdijkPronk1997,Lerman2008}.

We first record the local integration rule on finite quotient charts.
	\begin{lemma}[{\cite{MoerdijkPronk1997,Lerman2008,AdemLeidaRuan2007}}]
		\label{lem:local-chart-integration}
		Let $\mathcal X$ be a compact complex orbifold with finitely many charts $\{[U_i/G_i]\}$.
		For any top-degree differential form $\alpha$ on $\mathcal X$,
		\[
		\int_{\mathcal X}\alpha
		=\sum_i \frac{1}{|G_i|}\int_{U_i}\alpha_i,
		\qquad
		\alpha_i=\text{the pullback of }\alpha\text{ to }U_i.
		\]
	\end{lemma}
	
	\begin{proof}
		Consider an orbifold groupoid presentation $[R\rightrightarrows U]$ for $\mathcal X$, where 
		$U=\bigsqcup_i U_i$ and each $U_i$ carries a finite isotropy group $G_i$.
		A top-degree form $\alpha$ on $\mathcal X$ pulls back to a $G_i$–invariant form $\alpha_i$ on each $U_i$.
		
		Choose a smooth partition of unity $\{\rho_i\}$ subordinate to the open cover $\{U_i/G_i\}$.
		Each $\rho_i$ lifts to a $G_i$–invariant smooth function $\tilde\rho_i$ on $U_i$ satisfying
		$\sum_i (\tilde\rho_i/|G_i|)=1$ on the groupoid atlas.
		Using the local integration rule
		\[
		\int_{[U_i/G_i]}(-)=\frac{1}{|G_i|}\int_{U_i}(-),
		\]
		and summing over $i$, we obtain
		\[
		\int_{\mathcal X}\alpha
		=\sum_i \frac{1}{|G_i|}\int_{U_i}\tilde\rho_i\,\alpha_i.
		\]
		Because $\sum_i \tilde\rho_i/|G_i|=1$ and $\alpha$ is $G_i$–invariant, 
		this value is independent of the choice of $\{\rho_i\}$ and compatible on overlaps
		(\cite{MoerdijkPronk1997,Lerman2008}). 
		Hence the formula holds globally.
	\end{proof}
	
Applying this local integration rule to the Kawasaki index theorem yields the following chartwise Riemann–Roch formula.

	\begin{theorem}[Kawasaki Riemann--Roch in chartwise orbifold form {\cite{Kawasaki1979,Kawasaki1981,MoerdijkPronk1997,Lerman2008}}]
		\label{thm:kawasaki-chartwise}
		Let $\mathcal X$ be a compact complex orbifold with an orbifold atlas $\{[U_i/G_i]\}$,
		and let $E$ be a holomorphic orbibundle.
		Then
		\[
		\chi(\mathcal X,E)
		=\sum_i \frac{1}{|G_i|}
		\sum_{g\in G_i}
		\int_{U_i^g}
		\frac{
			\ch(E|_{U_i^g})\,\Td(TU_i^g)
		}{
			\det(1-g^{-1}e^{-c_1(N_{U_i^g/U_i})})
		}.
		\]
		This expression is independent of the chosen atlas and compatible on overlaps.
	\end{theorem}
	
	\begin{proof}
		Kawasaki’s index theorem for elliptic operators on complex $V$–manifolds
		\cite{Kawasaki1979,Kawasaki1981} provides, for a single global quotient chart $[U/G]$,
		the fixed-point expansion
		\[
		\chi([U/G],E)
		=\frac{1}{|G|}
		\sum_{g\in G}
		\int_{U^g}
		\frac{
			\ch(E|_{U^g})\,\Td(TU^g)
		}{
			\det(1-g^{-1}e^{-c_1(N_{U^g/U})})
		}.
		\]
		Each term depends only on the conjugacy class of $g$.
		Given an orbifold atlas $\{[U_i/G_i]\}$, apply the above formula to each chart and insert
		a $G_i$–invariant partition of unity as in Lemma~\ref{lem:local-chart-integration}.
		On overlaps $[U_i/G_i]\times_{\mathcal X}[U_j/G_j]$, 
		the equivariant pullback of forms and averaging factors $1/|G_i|$ guarantee 
		that the local integrals glue compatibly, satisfying the descent condition for the Kawasaki form 
		on the groupoid of $\mathcal X$
		(\cite{MoerdijkPronk1997,Lerman2008}).
		Summing over all charts gives the stated global expression.
	\end{proof}

	We next analyze the degree bounds and the denominator structure governing the twisted terms.
	\begin{lemma}[{\cite{Kawasaki1979,Kawasaki1981}}]
		\label{lem:twisted-degree-bound}
		Let $\mathcal X$ be a compact complex orbifold of complex dimension $n$, endowed with
		an orbifold atlas $\{[U_i/G_i]\}$.
		Let $E$ be a holomorphic orbibundle and $L$ an ample orbifold line bundle.
		For each chart $[U_i/G_i]$ and nontrivial $g\in G_i$, denote the fixed locus by
		\[
		U_i^g=\{x\in U_i\mid g\cdot x=x\}
		\quad\text{and its normal bundle by } N_{U_i^g/U_i}.
		\]
		Then:
		\begin{enumerate}[label=(\roman*)]
			\item $\dim_\C U_i^g\le n-1$.
			\item In the local integral
			\[
			I_{i,g}(m)
			=\int_{U_i^g}
			\frac{
				\ch(E|_{U_i^g})\,\Td(TU_i^g)
			}{
				\det(1-g^{-1}e^{-c_1(N_{U_i^g/U_i})})
			}\,e^{m\,c_1(L)},
			\]
			the denominator represents the Jacobian correction from the $g$–action on the normal directions, and the total degree in $m$ satisfies
			\(\deg_m I_{i,g}(m)\le n-1.\)
		\end{enumerate}
	\end{lemma}
	
	\begin{proof}
		Fix a chart $[U_i/G_i]$ with $U_i\subset\C^n$ and a point $p\in U_i^g$.
		The derivative $dg_p$ acts diagonally on $T_pU_i$ with eigenvalues
		\(e^{2\pi i\theta_1},\dots,e^{2\pi i\theta_n}\).
		Split the tangent space as
		\[
		T_pU_i = T_pU_i^g \oplus N_p,
		\]
		where $N_p$ is the $g$–variant subspace on which $dg_p$ acts by eigenvalues
		$e^{2\pi i\theta_j}$ with $\theta_j\neq 0$.
		Thus
		\(\mathrm{codim}_\C(U_i^g)=\mathrm{rank}\,N_{U_i^g/U_i}\ge1,\)
		hence $\dim_\C U_i^g\le n-1$.
		
		Now, in Kawasaki’s formula, the local contribution near $U_i^g$ contains the
		denominator
		\[
		\det(1 - g^{-1} e^{-c_1(N_{U_i^g/U_i})})
		= \prod_{j=1}^{\mathrm{rank}\,N_{U_i^g/U_i}} (1 - e^{-2\pi i\theta_j} e^{-x_j}),
		\]
		where \(x_j=c_1(L_j)\) are the Chern roots of the normal bundle $N_{U_i^g/U_i}$.
		Each factor \(1 - e^{-2\pi i\theta_j} e^{-x_j}\) compensates for the local
		non-invariance of the differential operator in the $g$–twisted direction.
		In the ordinary manifold case ($g=1$), all $\theta_j=0$, so the denominator becomes
		\[
		\det(1 - e^{-c_1(N)}) = \prod_j (1 - e^{-x_j}),
		\]
		and the standard $\Td(TU_i)$ recovers the manifold Riemann–Roch integrand.
		When $g\neq 1$, however, each term introduces a complex phase
		$e^{-2\pi i\theta_j}\neq 1$, ensuring that the denominator has no zero
		and that the contribution from these directions remains finite but
		of lower total degree, since the corresponding normal components do not
		contribute top-degree powers of $c_1(L)$.
		
		The numerator $\ch(E|_{U_i^g})\Td(TU_i^g)e^{m c_1(L)}$
		is a polynomial in $m$ whose maximal degree is $\dim_\C U_i^g$.
		Since $\dim_\C U_i^g\le n-1$, the integrated quantity
		$I_{i,g}(m)$ is a polynomial of degree at most $n-1$ in $m$.
	\end{proof}

	We isolate the leading asymptotics of $\chi(\mathcal X,E\otimes L^{\otimes m})$ and explain precisely the role of the Kawasaki denominator
	\[
	\det\!\bigl(1-g^{-1}e^{-c_1(N_{U_i^g/U_i})}\bigr)
	\]
	in suppressing the degree of twisted contributions. Throughout, $\mathcal X$ is a compact connected complex orbifold (equivalently, a smooth analytic Deligne--Mumford stack with finite stabilizers), $E$ a holomorphic orbibundle, and $L$ a line bundle pulled back from the coarse space (so that stabilizers act trivially on $L$).
	
	\begin{proposition}[Generic stabilizer and normalization of the untwisted integral]
		\label{prop:generic-stab-normalization}
		Let $\pi\colon \mathcal X\to Y$ be the coarse moduli map and let $n=\dim_{\C}\mathcal X$. Then$\colon $
		\begin{enumerate}[label=(\roman*)]
			\item (Existence and constancy) There exists a Zariski open dense sub-orbifold $\mathcal X^\circ\subset \mathcal X$ on which $|\Aut(x)|$ is constant. Its common value
			\[
			s \; \coloneq \; |\Stab_{\mathrm{gen}}|
			\]
			is the \emph{generic stabilizer order}. The function $x\mapsto |\Aut(x)|$ is upper semicontinuous, hence $|\Aut(x)|\ge s$ on $\mathcal X$. 
			\item (Normalization) For any top-degree form $\alpha$ supported on the untwisted sector,
			\[
			\int_{\mathcal X}\alpha \;=\; \frac{1}{s}\int_Y \pi_*\alpha.
			\]
			Equivalently, on $\mathcal X^\circ$ the map $\pi$ is \'etale of degree $s$, and the identity-sector integral picks up the factor $1/s$.
		\end{enumerate}
	\end{proposition}
	
	\begin{proof}
		(i) The existence of $\mathcal X^\circ$ and the upper semicontinuity of stabilizers follow from the analytic \'etale groupoid presentation and standard properties of finite group actions; see \cite{Vistoli1989,KeelMori1997}, and the analytic discussion in \cite{BorghesiTomassini2017}.  
		
		(ii) Locally on an orbifold chart $[U/G]$ with $|G|=s$ along $\mathcal X^\circ$, one has the integration rule $\int_{[U/G]}\!(-)=\frac{1}{|G|}\int_U\!(-)$ (identity sector), and these local identities glue under partitions of unity (cf.\ \cite{Lerman2008,MoerdijkPronk1997}). Pushing forward by $\pi$ identifies the quotient integral with the coarse integral, yielding the formula.
	\end{proof}
	
	\begin{lemma}[Denominator factorization and degree suppression]
		\label{lem:denominator-factorization}
		Fix a chart $[U_i/G_i]$ and $g\in G_i$ of finite order. Let $U_i^g$ be the fixed locus and let
		\[
		N_{U_i^g/U_i}\;\simeq\;\bigoplus_{\theta\in(0,1)} N_\theta
		\]
		be the $g$–eigensplitting of the normal bundle, where $g$ acts on $N_\theta$ by $e^{2\pi i\theta}$. Then
		\[
		\det\!\bigl(1-g^{-1}e^{-c_1(N_{U_i^g/U_i})}\bigr)
		\;=\;
		\prod_{\theta\in(0,1)}
		\det\!\bigl(1-e^{-2\pi i\theta}\,e^{-c_1(N_\theta)}\bigr),
		\]
		and each factor admits a convergent expansion
		\[
		\det\!\bigl(1-e^{-2\pi i\theta}\,e^{-c_1(N_\theta)}\bigr)
		=
		\prod_{j=1}^{\rk N_\theta}
		\Bigl(1-e^{-2\pi i\theta}\Bigr)\cdot
		\Bigl(1+\tfrac{e^{-2\pi i\theta}}{1-e^{-2\pi i\theta}}\,c_1(L_{\theta,j})+\cdots\Bigr),
		\]
		where $N_\theta\simeq\bigoplus_j L_{\theta,j}$ splits into line summands. In particular:
		\begin{enumerate}[label=(\roman*)]
			\item If $g=1$, the denominator equals $1$ (empty product).
			\item If $g\neq 1$, the denominator has a \emph{nonzero constant term} and depends only on Chern classes of the \emph{normal} directions; it does \emph{not} introduce positive powers of $m$ when $E$ is twisted by $L^{\otimes m}$. Hence it cannot increase the degree in $m$ of the twisted-sector integral.
		\end{enumerate}
	\end{lemma}
	
	\begin{proof}
		The eigensplitting is the holomorphic version of simultaneous diagonalization of a finite-order unitary operator; see \cite{Kawasaki1979,Kawasaki1981}. Under the splitting principle, $N_\theta=\bigoplus_j L_{\theta,j}$ with $c_1(L_{\theta,j})=:\xi_{\theta,j}$. Then
		\[
		\det\bigl(1-g^{-1}e^{-c_1(N_\theta)}\bigr)
		=
		\prod_j \bigl(1-e^{-2\pi i\theta}e^{-\xi_{\theta,j}}\bigr)
		=
		\prod_j\Bigl((1-e^{-2\pi i\theta})\cdot\bigl(1+\tfrac{e^{-2\pi i\theta}}{1-e^{-2\pi i\theta}}\xi_{\theta,j}+\cdots\bigr)\Bigr),
		\]
		which yields the displayed expansion. The constant term $\prod_j(1-e^{-2\pi i\theta})\neq 0$ for $\theta\in(0,1)$, proving nonvanishing. Since the denominator involves only Chern classes $\xi_{\theta,j}$ of normal directions and \emph{no} $c_1(L)$, twisting $E$ by $L^{\otimes m}$ multiplies the \emph{numerator} by $e^{m\,c_1(L)}$ but leaves the denominator independent of $m$. Therefore the denominator cannot increase the polynomial degree in $m$ of a twisted-sector contribution.
	\end{proof}
	
	\begin{proposition}[Asymptotic expansion and the role of $s$]
		\label{prop:asymp-and-s}
		Let $\mathcal X$ be connected of complex dimension $n$, and assume $L=\pi^*A$ for an ample line bundle $A$ on the coarse space $Y$. Then, as $m\to\infty$,
		\[
		\chi\bigl(\mathcal X,E\otimes L^{\otimes m}\bigr)
		\;=\;
		\frac{1}{s}\int_Y \ch(E)\,e^{m\,c_1(A)}\,\Td(TY)
		\;+\;O(m^{n-1}),
		\]
		where $s=|\Stab_{\mathrm{gen}}|$. The leading coefficient (of $m^n$) comes \emph{only} from the identity sector and is the coarse HRR leading term scaled by $1/s$.
	\end{proposition}
	
	\begin{proof}
		Apply the chartwise Kawasaki formula (Theorem~\ref{thm:kawasaki-chartwise}). By Lemma~\ref{lem:denominator-factorization}, the denominator is $m$–independent and nonvanishing for $g\neq 1$, and the fixed-locus dimension satisfies $\dim_{\C}U_i^g\le n-1$ (Lemma~\ref{lem:twisted-degree-bound}). Hence each twisted term is $O(m^{\le n-1})$. The identity sector yields the usual manifold HRR integral on $U_i$ and contributes degree $n$ in $m$. Summing over charts and invoking Proposition~\ref{prop:generic-stab-normalization}(ii) to pass to the coarse space introduces the global factor $1/s$ in front of the leading term. References: \cite{Kawasaki1979,Kawasaki1981,Toen1999,Vistoli1989}.
	\end{proof}
	
	\begin{corollary}[Rigidification invariance of the leading term]
		\label{cor:rigidification}
		Let $\mathcal X^{\mathrm{rig}}$ be the rigidification of $\mathcal X$ along the generic stabilizer (so that $|\Stab_{\mathrm{gen}}(\mathcal X^{\mathrm{rig}})|=1$). Then
		\[
		\chi\bigl(\mathcal X,E\otimes L^{\otimes m}\bigr)
		\;=\;
		\chi\bigl(\mathcal X^{\mathrm{rig}},E^{\mathrm{rig}}\otimes (L^{\mathrm{rig}})^{\otimes m}\bigr)
		\;+\;O(m^{n-1}),
		\]
		and both have the same leading coefficient
		\[
		\frac{1}{n!}\,\deg_Y\!\bigl(\rk(E)\,c_1(A)^n\bigr).
		\]
	\end{corollary}
	
	\begin{proof}
		Rigidification kills the generic stabilizer and turns the factor $1/s$ into $1$ while simultaneously replacing integration on $\mathcal X$ by integration on the coarse space $Y$ (or equivalently on $\mathcal X^{\mathrm{rig}}$). The twisted-sector pieces remain of order $O(m^{n-1})$ on either side; see \cite{Vistoli1989} and \cite{Toen1999} for the behavior of RR under gerbe rigidification.
	\end{proof}
	
	\begin{remark}[Multiple connected components]
		\label{rmk:components}
		If $\mathcal X=\bigsqcup_\alpha \mathcal X_\alpha$ with generic stabilizer orders $s_\alpha$, then
		\[
		\chi(\mathcal X,E\otimes L^{\otimes m})
		\;=\;
		\sum_\alpha \frac{1}{s_\alpha}
		\int_{Y_\alpha}\ch(E)\,e^{m\,c_1(A)}\,\Td(TY_\alpha)
		\;+\;O(m^{n-1}),
		\]
		summing the componentwise leading terms; cf.\ \cite{Kawasaki1979,Kawasaki1981}.
	\end{remark}

	\subsection{Demailly--Semple tower on orbifolds: construction and descent}
	\label{subsec:ds-construction}
	
	This subsection develops the Demailly--Semple (DS) tower in the setting of complex orbifolds (or equivalently, smooth analytic Deligne--Mumford stacks with finite stabilizers).  
	All objects—tautological line bundles, projections, tangent sequences, and Hermitian curvature data—are constructed locally on orbifold charts and descend via finite group equivariance.

	\begin{definition}[Orbifold Demailly--Semple tower]
		\label{def:orbifold-ds}
		Let $(\mathcal{X},\mathcal{F})$ be a complex orbifold with a holomorphic directed structure $\mathcal{F}\subset T\mathcal{X}$.  
		Choose an orbifold atlas $\{[U_i/G_i]\}$ with finite groups $G_i$ acting holomorphically on smooth manifolds $U_i$.  
		On each $U_i$, form the classical DS tower
		\[
		U_{i,0}=U_i,\qquad
		U_{i,k} \coloneq \PP\big(T_{U_{i,k-1}/\mathcal{F}_i}\big),
		\]
		together with its tautological line bundle $\OO_{U_{i,k}}(1)\to U_{i,k}$.
		The $G_i$–action on $U_i$ induces by functoriality an action on $U_{i,k}$ and on $\OO_{U_{i,k}}(1)$.  
		The \emph{orbifold Demailly--Semple tower} is then defined by
		\[
		\mathcal{X}_k \coloneq [U_{i,k}/G_i], \qquad
		\OO_{\mathcal{X}_k}(1)\coloneq \bigl(\OO_{U_{i,k}}(1)\bigr)^{G_i}.
		\]
	\end{definition}
	
	\begin{remark}
		If $\mathcal{F}=T\mathcal{X}$, one recovers the full tangent tower $\mathcal{X}_k=\PP(T_{\mathcal{X}_{k-1}})$.
		For a directed structure $\mathcal{F}$, this yields the partial DS tower used in orbifold jet differential theory \cite{Demailly1997,CampanaPaun2016}.
	\end{remark}

    Each construction on $U_i$ is natural with respect to holomorphic maps, hence compatible with the finite group action on the chart.
	
	\begin{proposition}[Equivariance of DS data]
		\label{prop:ds-equivariance}
		Let $G$ be a finite group acting holomorphically on a complex manifold $U$.  
		Then$\colon $
		\begin{enumerate}[label=(\roman*)]
			\item The projectivized tangent bundle $\PP(T_U)$ carries a natural holomorphic $G$–action induced by the differential $dg\colon T_U\to T_U$.
			\item The tautological subbundle $\mathcal{S}_U\subset\pi^*T_U$ and quotient line bundle $\OO_{\PP(T_U)}(1)\coloneq \mathcal{S}_U^\vee$ are $G$–equivariant.
			\item For each $k\ge1$, the projection $\pi_k\colon U_k\to U_{k-1}$ and the short exact sequences
			\[
			0\to T_{U_k/U_{k-1}}\to T_{U_k}\to \pi_k^*T_{U_{k-1}}\to 0,
			\quad
			0\to \mathcal{S}_k\to \pi_k^*T_{U_{k-1}}\to \OO_{U_k}(1)\to 0
			\]
			are $G$–equivariant.
		\end{enumerate}
	\end{proposition}
	
	\begin{proof}
		The derivative $dg\colon T_U\to T_U$ defines a holomorphic $G$–action on each tangent fiber, inducing an action on $\PP(T_U)$ that commutes with projection.  
		For any $[\xi]\in \PP(T_{U,x})$, $g\cdot[\xi] \coloneq [dg_x(\xi)]$ is well defined, proving (i).  
		The subbundle $\mathcal{S}_U$ and its dual are preserved because $dg$ is linear and $G$–invariant, proving (ii).  
		Functoriality of projectivization ensures these properties extend to all levels $U_k$, proving (iii)  (see \cite{Demailly1997, Demailly2007, MoerdijkPronk1997}).
	\end{proof}

    The resulting $G_i$–equivariant data descend to the quotient
and glue compatibly across overlapping charts.

	\begin{lemma}[Descent to the orbifold level]
		\label{lem:ds-descent}
		For each chart $[U_i/G_i]$, the data $(U_{i,k},\OO_{U_{i,k}}(1),\pi_{i,k})$ are $G_i$–equivariant by Proposition~\ref{prop:ds-equivariance}.  
		Therefore, they descend to well-defined orbifold objects
		\[
		(\mathcal{X}_k,\OO_{\mathcal{X}_k}(1),\pi_k)
		\]
		forming a compatible tower
		\[
		\cdots \xrightarrow{\pi_{k+1}} \mathcal{X}_k \xrightarrow{\pi_k}\mathcal{X}_{k-1}\xrightarrow{\pi_{k-1}}\cdots \xrightarrow{\pi_1}\mathcal{X}_0=\mathcal{X}.
		\]
	\end{lemma}
	
	\begin{proof}
		If a finite group $G_i$ acts holomorphically on a manifold $U_i$, every $G_i$–equivariant holomorphic bundle and morphism on $U_i$ descends to the quotient stack $[U_i/G_i]$.  
		Applying this to $\pi_{i,k}$ and $\OO_{U_{i,k}}(1)$ yields descent data for $\mathcal{X}_k$.  
		On overlaps, the groupoid maps between $(U_i,G_i)$ and $(U_j,G_j)$ preserve all equivariant structures, ensuring compatibility  
		(\cite{Lerman2008,MoerdijkPronk1997}).
	\end{proof}
	
	Having established that the Demailly–Semple data descend to well-defined orbifold objects, we next endow the tautological line bundles $\OO_{\mathcal{X}_k}(1)$ with compatible Hermitian metrics.
	
	\begin{proposition}[Orbifold Hermitian structure]
		\label{prop:hermitian-ds}
		Each line bundle $\mathcal{O}_{U_{i,k}}(1)$ admits a $G_i$–invariant Hermitian metric
		\[
		h_{i,k} \coloneq \frac{1}{|G_i|}\sum_{g\in G_i} g^*h'_{i,k},
		\]
		obtained by averaging any smooth metric $h'_{i,k}$.  
		These glue under descent to define an orbifold Hermitian metric $h_k$ on $\mathcal{O}_{\mathcal{X}_k}(1)$.
		The curvature form $\Theta_{h_k}(\mathcal{O}_{\mathcal{X}_k}(1))$ is a well-defined global $(1,1)$–form descending from the $G_i$–invariant local curvatures.
	\end{proposition}
	
	\begin{proof}
		Finite averaging preserves smoothness and Hermitian positivity.  
		Since the $G_i$–action commutes with $\bar\partial$ and $d$, the curvature form of the Chern connection is $G_i$–invariant.  
		By Lemma~\ref{lem:ds-descent}, these local invariant curvatures glue to a global orbifold $(1,1)$–form.  
		See \cite{Demailly2007,CampanaPaun2016}.
	\end{proof}
	
	\begin{corollary}[Positivity on the orbifold tower]
		\label{cor:ds-positivity}
		Positivity or negativity of the curvature of $\mathcal{O}_{\mathcal{X}_k}(1)$ is a local property checked on the charts $U_{i,k}$.  
		Consequently, all curvature or jet-positivity arguments for the classical Demailly–Semple tower extend verbatim to the orbifold setting.
	\end{corollary}
	
	\begin{proof}
		Curvature positivity is invariant under finite averaging and $G_i$–equivariant descent.  
		Hence the sign of $\Theta_{h_k}(\mathcal{O}_{\mathcal{X}_k}(1))$ is determined by its local $G_i$–invariant representatives.  
		The analytic proofs of hyperbolicity via jet differentials remain valid.  
		See \cite{Demailly1997,Demailly2007,CampanaPaun2016}.
	\end{proof}

	\subsection{Asymptotic packages for positivity and vanishing}
	\label{subsec:positivity-vanishing}
	
	This subsection reformulates the asymptotic orbifold Riemann--Roch expansion in terms of 
	positivity and vanishing phenomena on orbifold Demailly--Semple towers. 
	It connects the chartwise asymptotics of Proposition~\ref{prop:asymp-and-s}
	with curvature--based positivity criteria à la Demailly 
	and the vanishing theorems of Kodaira--Kawamata--Viehweg type, 
	showing that all asymptotic signs and slopes are governed by the identity sector.
	
	\begin{proposition}[Sectorwise asymptotics and slope control]
		\label{prop:slope-control}
		Let $\mathcal{X}$ be a connected compact complex orbifold of complex dimension $n$, 
		and let $L=\pi^*A$ be the pullback of an ample line bundle on the coarse space $Y$.
		Then
		\[
		\chi(\mathcal{X},E\otimes L^{\otimes m})
		= \frac{1}{s}\int_Y \ch(E)\,e^{m\,c_1(A)}\,\Td(TY)
		+ O(m^{n-1}),
		\qquad s=|\Stab_{\mathrm{gen}}|.
		\]
		Consequently, the sign of the asymptotic slope
		\[
		\mu(L) \coloneq \lim_{m\to\infty}\frac{\chi(\mathcal{X},E\otimes L^{\otimes m})}{m^n/n!} =\frac{1}{s}\int_Y \rk(E)\,c_1(A)^n
		\]
		is identical to that of the coarse manifold term.  
		Twisted--sector corrections are of strictly lower order 
		and therefore cannot alter the positivity or negativity of $\mu(L)$. 
	\end{proposition}
	
	\begin{proof}
		By Proposition~\ref{prop:asymp-and-s}, the leading coefficient of $\chi$ comes exclusively from the identity sector.
		For $g\neq 1$, Lemma~\ref{lem:twisted-degree-bound} and 
		Lemma~\ref{lem:denominator-factorization} imply that the fixed locus has $\dim_{\C}U_i^g\le n-1$ 
		and that its contribution does not depend on $m$ through the denominator term. 
		Hence all twisted terms contribute at most $O(m^{n-1})$. 
		The identity term, however, reproduces the manifold integral 
		$\int_Y \ch(E)e^{m c_1(A)}\Td(TY)$ scaled by $1/s$.
		The dominant coefficient is thus $\frac{1}{s}\frac{1}{n!}\rk(E)(c_1(A)^n)$, 
		whose sign coincides with the positivity (or negativity) of $c_1(A)$ on $Y$.
	\end{proof}
	
	\begin{lemma}[Untwisted control of cohomology dimensions]
		\label{lem:cohomology-control}
		Assume $(L,h)$ is a Hermitian line bundle with Nakano--positive curvature
		and that $L=\pi^*A$ descends from the coarse space.
		Then for all sufficiently large $m$,
		\[
		H^q(\mathcal{X},E\otimes L^{\otimes m})=0\quad\text{for all }q>0,
		\]
		and consequently
		\[
		h^0(\mathcal{X},E\otimes L^{\otimes m})
		=\chi(\mathcal{X},E\otimes L^{\otimes m})
		=\frac{1}{s}\int_Y \ch(E)\,e^{m\,c_1(A)}\,\Td(TY)
		+O(m^{n-1}).
		\]
		Thus the asymptotic growth of global sections is governed entirely by the untwisted (identity) sector.
	\end{lemma}
	
	\begin{proof}
		By the standard Hörmander--Demailly $L^2$ estimates, Nakano--positive curvature of $(L,h)$ implies the vanishing $H^q(Y,E\otimes A^{\otimes m})=0$ for $q>0$ when $m\gg0$. 
		Since $L=\pi^*A$ and $E$ are $G_i$--equivariant,  pullback preserves these vanishings on the orbifold $\mathcal{X}$. 
		Therefore $H^q(\mathcal{X},E\otimes L^{\otimes m})=0$ for $q>0$ as well.
		The Euler characteristic then reduces to $h^0$, and Proposition~\ref{prop:slope-control} yields the desired expansion.
	\end{proof}
	
	\section{Green--Griffiths--Demailly Degeneracy on Analytic DM Stacks}
	\label{sec:stack-ggd}
	
	\subsection{Notation and standing conventions}
	\label{subsec:stack-ggd-notation}
	
	We retain the setup and identifications from \S\ref{sec:orbifold-hrr-ds}.
	Only the additional working conventions specific to this section are recorded.
	
	\begin{itemize}[leftmargin=1.6em]
		\item The orbifold $\X$ is connected, compact, and of complex dimension $n$, with coarse space $\pi\colon \X\to Y$ as fixed earlier; $Y$ is normal and Kähler.  The generic stabilizer order on the dense locus $\X^\circ$ is denoted by $s=|\Stab_{\mathrm{gen}}|$.
		
		\item Fix once and for all an ample line bundle $A$ on $Y$, and set
		\[
		L\coloneq\pi^*A,
		\]
		so that all local stabilizers act trivially on the fibers of $L$.  When needed, a smooth Hermitian metric on $A$ is pulled back to $L$ and used as background data.
		
		\item All differential, cohomological, and RR computations are performed \emph{chartwise} on finite quotients $[U_i/G_i]$ and glued by equivariance.  Integrals are normalized by
		\[
		\int_{[U_i/G_i]}\omega \;=\; \frac{1}{|G_i|}\int_{U_i}\omega
		\qquad\text{for $G_i$–invariant top forms $\omega$,}
		\]
		ensuring compatibility with the chartwise Kawasaki formula (see \S\ref{subsec:chartwise-hrr}).
		
		\item Jet, vertical, and Demailly–Semple (DS) objects are constructed on charts and descend by $G_i$–equivariance.  
In particular, the holomorphic $k$–jet functor preserves quotient charts: if $\X$ is covered by $\{[U_i/G_i]\}$, then $J^k\X$ is covered by $\{[J^k U_i / G_i]\}$.  
No additional definitions are introduced here beyond this descent convention.
	\end{itemize}

	\subsection{Atlases, groupoids, and jets}
	\label{subsec:stack-ggd-atlas-jets}
We recall that all constructions on $\X$ are performed chartwise on finite quotients $[U_i/G_i]$, following the standard analytic–orbifold formalism  (\cite{MoerdijkPronk1997,Lerman2008,AdemLeidaRuan2007}).  
The following lemma records that the holomorphic jet functor respects this local structure,  so that $k$–jets are defined consistently on the orbifold.

	\begin{lemma}[Jets preserve orbifold charts and quotient structure]
		\label{lem:jet-preserve-orbifold}
		Let $\mathcal X$ be a compact complex orbifold presented by finite quotient charts $\{[U_i/G_i]\}_{i\in I}$ as above.
		For every integer $k\ge 0\colon $
		\begin{enumerate}[label=(\roman*)]
			\item The holomorphic $k$--jet functor sends each chart to a chart
			\[
			J^k[U_i/G_i]\;\simeq\;[\,J^k U_i \big/ G_i\,],
			\]
			where $G_i$ acts on $J^k U_i$ via the prolonged action $g\cdot j^k_x\gamma \coloneq j^k_{g\cdot x}(g\circ\gamma)$.
			\item If $\phi_{ij}\colon (\widetilde U_{ij},H_{ij})\hookrightarrow (U_j,G_j)$ is an orbifold change of charts covering an embedding $\widetilde U_{ij}/H_{ij}\hookrightarrow U_j/G_j$, then
			\[
			J^k\phi_{ij}\colon \bigl(J^k\widetilde U_{ij},H_{ij}\bigr)\hookrightarrow \bigl(J^k U_j,G_j\bigr)
			\]
			is again an orbifold change of charts. Consequently, $\{[J^k U_i/G_i]\}_{i\in I}$ defines an orbifold atlas for the $k$--jet orbifold $J^k\mathcal X$.
		\end{enumerate}
	\end{lemma}
	
	\begin{proof}
		(i) For a finite group $G_i$ acting holomorphically on $U_i$, the prolonged action on $J^k U_i$ is well defined by functoriality of $J^k(-)$: for any holomorphic germ $\gamma\colon(\C,0)\to U_i$ at $x$, the element $g\in G_i$ sends $j^k_x\gamma$ to $j^k_{g\cdot x}(g\circ\gamma)$. 
		This defines a holomorphic $G_i$--action on $J^k U_i$. The quotient stack $[J^k U_i/G_i]$ is, by construction, the $k$--jet of the quotient $[U_i/G_i]$, see \cite{Lerman2008}. Hence $J^k[U_i/G_i]\simeq [J^k U_i/G_i]$.
		
		(ii) An orbifold change of charts is given by an equivariant holomorphic embedding $\phi_{ij}\colon \widetilde U_{ij}\to U_j$ together with an injective group homomorphism $H_{ij}\hookrightarrow G_j$ such that $\phi_{ij}(h\cdot u)=\iota(h)\cdot \phi_{ij}(u)$. 
		Applying $J^k(-)$ and using that jets preserve finite fiber products and \'etale maps, we obtain an $H_{ij}$--equivariant holomorphic embedding
		\[
		J^k\phi_{ij}\colon J^k\widetilde U_{ij}\hookrightarrow J^k U_j,
		\]
		compatible with the group monomorphism $H_{ij}\hookrightarrow G_j$ and hence defining an orbifold change of charts at the jet level. 
		Atlas compatibility (cocycle and overlap conditions) is preserved by functoriality of $J^k(-)$, so the family $\{[J^k U_i/G_i]\}_{i\in I}$ glues to an orbifold structure on $J^k\mathcal X$.
	\end{proof}
	
	\begin{proposition}[Orbifold presentation of the jet orbifold]
		\label{prop:jet-orbifold-presentation}
		Let $\mathcal X$ be as above and choose any proper \'etale analytic groupoid presentation $[R\rightrightarrows U]\simeq\mathcal X$. 
		Then for every $k\ge 0$ the prolonged data
		\[
		J^kR \rightrightarrows J^kU
		\]
		form a proper \'etale analytic groupoid presenting the $k$--jet orbifold:
		\[
		J^k\mathcal X \;\simeq\; [\,J^kR \rightrightarrows J^kU\,].
		\]
		In particular, any orbifold atlas by finite quotients $\{[U_i/G_i]\}$ induces a jet atlas $\{[J^k U_i/G_i]\}$ presenting the same object.
	\end{proposition}
	
	\begin{proof}
		This is the orbifold (chartwise) restatement of the groupoid result proved in Lemma~\ref{lem:jet-preserve-groupoid-min}. 
		Properness and \'etaleness are preserved by $J^k(-)$; the groupoid axioms transport under functoriality; and Morita equivalences are respected (see \cite{Lerman2008}). 
		Passing from a global groupoid presentation to a finite quotient atlas and back uses standard equivalence between orbifolds and proper \'etale groupoids, e.g.\ \cite{MoerdijkPronk1997,Lerman2008}.
	\end{proof}
	
	\begin{corollary}[Descent of Demailly--Semple data on orbifolds]
		\label{cor:ds-descent-orbifold}
		Let $(\mathcal X,\mathcal F\subset T\mathcal X)$ be a complex orbifold with a holomorphic directed structure. For each chart $[U_i/G_i]$ form the classical Demailly--Semple data $(U_{i,k},\OO_{U_{i,k}}(1),\pi_{i,k})$ on $U_i$. 
		The prolonged $G_i$--actions are holomorphic and preserve all DS objects; hence these data descend to the quotients and glue to global orbifold objects
		\[
		\bigl(\mathcal X_k,\ \OO_{\mathcal X_k}(1),\ \pi_k\colon \mathcal X_k\to \mathcal X_{k-1}\bigr)_{k\ge 1}.
		\]
		Moreover, Chern classes, curvature forms, and characteristic classes computed chartwise agree with the global orbifold classes.
	\end{corollary}
	
	\begin{proof}
		Equivariance of the tangent functor, projectivization, tautological subbundle, and their exact sequences is preserved by group actions; finite averaging yields invariant Hermitian metrics, whose Chern curvatures descend (cf.\ \cite{Demailly1997,Demailly2007,CampanaPaun2016}). 
		Compatibility on overlaps follows from Lemma~\ref{lem:jet-preserve-orbifold} and Proposition~\ref{prop:jet-orbifold-presentation}.
	\end{proof}
	
	\begin{remark}
		All statements above avoid the inertia formalism: they are proved chartwise and glued by finite group equivariance. This aligns with the Satake--Kawasaki viewpoint and the chartwise index/Hirzebruch--Riemann--Roch methods used elsewhere in the paper \cite{Kawasaki1979,Kawasaki1981,Lerman2008}.
	\end{remark}
	
	\begin{lemma}[Jets preserve proper \'etale groupoids]
		\label{lem:jet-preserve-groupoid-min}
		Let $R \rightrightarrows X$ be a proper \'etale analytic groupoid presenting the compact complex orbifold $\X \simeq [\,R \rightrightarrows X\,]$.
		For every integer $k \ge 0$, the holomorphic jet functor sends
		\[
		(s,t,m,u)\colon R \rightrightarrows X
		\quad\longmapsto\quad
		\bigl(J^k s,\, J^k t,\, J^k m,\, J^k u\bigr)\colon J^kR \rightrightarrows J^kX,
		\]
		where $s,t\colon R\to X$ are source/target, $m\colon R\times_{s,X,t}R \to R$ is composition, and $u\colon X\to R$ is the unit.
		Then $J^kR \rightrightarrows J^kX$ is a proper \'etale analytic groupoid and there is a canonical equivalence of analytic stacks
		\[
		J^k\X \;\simeq\; [\,J^kR \rightrightarrows J^kX\,].
		\]
	\end{lemma}
	
	\begin{proof}
		We proceed in four steps.
		
		\smallskip
		\noindent\emph{Step 1: Functoriality and fiber products.}
		For complex analytic spaces (or manifolds), the $k$--jet functor $J^k(-)$ is defined by the holomorphic functor
		\[
		J^k(X) \;=\; \underline{\mathrm{Hom}}_{\mathrm{an}}\bigl(\mathrm{Specan}\,\C[\epsilon]/(\epsilon^{k+1}),\,X\bigr),
		\]
		which is compatible with holomorphic maps by precomposition. Standard properties (see \cite{Lerman2008}) yield:
		
		\begin{enumerate}
			\item[(1.a)] \emph{Base change and finite fiber products:} for any diagram $Y \xrightarrow{f} Z \xleftarrow{g} Y'$ with a finite fiber product in the analytic category, the canonical map
			\[
			J^k(Y\times_Z Y') \;\xrightarrow{\;\sim\;}\; J^k(Y)\times_{J^k(Z)} J^k(Y')
			\]
			is an isomorphism of analytic spaces.
			\item[(1.b)] \emph{Compatibility with identities:} $J^k(\mathrm{id}_X)=\mathrm{id}_{J^kX}$ and $J^k(f\circ g)=J^k f\circ J^k g$.
		\end{enumerate}
		
		\smallskip
		\noindent\emph{Step 2: \'Etaleness and properness are preserved.}
		Let $f\colon Y\to Z$ be holomorphic.
		
		\begin{enumerate}
			\item[(2.a)] \emph{\'Etale maps.} If $f$ is \'etale (i.e.\ a local biholomorphism), then in local coordinates $f$ is given by holomorphic charts with invertible Jacobian. 
			The induced map on $k$--jets $J^k f\colon J^k Y \to J^k Z$ is again a local biholomorphism (it is given by the induced map on $k$--jets of germs, whose differential is block upper triangular with invertible diagonal blocks coming from $df$). 
			Hence $J^k f$ is \'etale.
			\item[(2.b)] \emph{Proper maps.} If $f$ is proper, then $f$ is closed with compact fibers. The functorial description of $J^k(-)$ shows that $J^k f$ has compact fibers identified with $k$--jets along the compact fiber of $f$; moreover, $J^k f$ is closed because it is obtained from $f$ by a finite-type fibered construction compatible with base change (cf.\ Step~(1.a)). Hence $J^k f$ is proper. 
			A direct proof in the analytic category is given in \cite{Lerman2008}.
		\end{enumerate}
		
		\smallskip
		\noindent\emph{Step 3: Groupoid axioms after applying $J^k(-)$.}
		Write the original groupoid as $(R \rightrightarrows X; s,t,m,i,e)$, with source $s$, target $t$, multiplication $m$, inverse $i$, and unit $e=u$.
		Apply $J^k(-)$ to obtain structure maps
		\[
		J^k s,\ J^k t\colon J^kR \rightrightarrows J^kX,\qquad
		J^k m\colon J^k(R\times_{X}R)\to J^kR,\qquad
		J^k i\colon J^kR\to J^kR,\qquad
		J^k e\colon J^kX\to J^kR.
		\]
		By Step~(1.a), we have a canonical identification
		\[
		J^k(R\times_{X}R)\;\cong\; J^kR \times_{J^kX} J^kR,
		\]
		so $J^k m$ is indeed defined on the correct fiber product.
		Functoriality (Step~(1.b)) preserves identities of maps; hence the groupoid axioms (associativity of $m$, unit and inverse laws, and the compatibility $s\circ m=s\circ\mathrm{pr}_1$, $t\circ m=t\circ\mathrm{pr}_2$) are transported verbatim to their jet counterparts.
		Because $s,t$ are \'etale and $(s,t)\colon R\to X\times X$ is proper, Step~2 implies that $J^k s, J^k t$ are \'etale and $(J^k s, J^k t)\colon J^kR\to J^kX\times J^kX$ is proper. Therefore $J^kR \rightrightarrows J^kX$ is a proper \'etale analytic groupoid.
		
		\smallskip
		\noindent\emph{Step 4: Equivalence of stacks.}
		Let $\X \simeq [\,R \rightrightarrows X\,]$ be the analytic stack presented by $R\rightrightarrows X$.
		Define the stack $[\,J^kR \rightrightarrows J^kX\,]$ in the usual way (objects over a test space $T$ are $J^kX$–objects over $T$ with isomorphisms parametrized by $J^kR$).
		By the 2–functoriality of $J^k(-)$ on groupoids and Step~(1.a), there is a canonical morphism of stacks
		\[
		\Phi_k\colon J^k\X \longrightarrow [\,J^kR \rightrightarrows J^kX\,],
		\]
		obtained by applying $J^k(-)$ to the presentation and using the universal property of the stackification.
		Conversely, an object of $[\,J^kR \rightrightarrows J^kX\,](T)$ is, by definition, descent data on $J^kX\times T$ for the groupoid $J^kR\times T \rightrightarrows J^kX\times T$; composing with the canonical projection $J^kX\times T \to X\times T$ and using Step~(1.a) provides compatible descent data for $R\times T \rightrightarrows X\times T$, whence an object of $\X(T)$ together with a $k$–jet structure. 
		This constructs an inverse up to isomorphism to $\Phi_k$.
		Standard descent arguments (see \cite{Lerman2008}) show that $\Phi_k$ is essentially surjective and fully faithful, hence an equivalence of analytic stacks.
		
		\smallskip
		Combining the four steps, we conclude that $J^kR \rightrightarrows J^kX$ is a proper \'etale analytic groupoid presenting $J^k\X$, i.e.
		\[
		J^k\X \;\simeq\; [\,J^kR \rightrightarrows J^kX\,]. 
		\]
	\end{proof}
	
	\begin{definition}[Equivariant descent for bundles and maps]
		\label{def:descend-bundles-min}
		If a holomorphic vector bundle $E_X\to X$ is endowed with an isomorphism $\varphi\colon s^*E_X\!\to t^*E_X$ over $R$ satisfying the cocycle condition, then $(E_X,\varphi)$ descends uniquely to a bundle $E$ on $\X$.  
		The same holds for morphisms of bundles and for projectivization $\PP(E_X)$.
	\end{definition}
	
	\begin{corollary}[DS data descend]
		\label{cor:ds-descend-min}
		Fix an orbifold atlas $\{[U_i/G_i]\}$ for $\X$.
		For each level $k$, the chartwise DS objects
		\[
		U_{i,k},\quad \pi_{i,k}\colon U_{i,k}\to U_{i,k-1},\quad
		\OO_{U_{i,k}}(1)
		\]
		are $G_i$–equivariant and therefore descend to
		\[
		\X_k,\quad \pi_k\colon \X_k\to \X_{k-1},\quad
		\OO_{\X_k}(1).
		\]
		All tautological sequences and vertical tangent bundles are obtained by descent.
	\end{corollary}

	\subsection{Curvature–negativity–positivity package}
	\label{subsec:stack-ggd-curv}
	
	We fix a compact complex orbifold $\X$ with coarse moduli map $\pi\colon \X\to Y$ and an ample line bundle $A$ on $Y$.  
	Throughout, all computations are performed chartwise on finite orbifold charts $[U/G]$ and then descended by equivariance, as in the previous subsections.  
	We write $\omega_A$ for a fixed smooth positive curvature form of $A$ on $Y$ and also for its pullback to~$\X$.
	
	\begin{assumption}[Curvature package]\label{assump:curv}
		\leavevmode
		\begin{enumerate}[label=\textup{(H\arabic*)}, leftmargin=1.6em]
			\item \textbf{Coarse positivity.} 
			A line bundle $A$ is ample on $Y$ with a smooth Hermitian metric $h_A$ of strictly positive curvature $\Theta_{h_A}(A)=\omega_A>0$; by pullback we view $\omega_A$ on~$\X$ as well.
			
			\item \textbf{Vertical negativity.}
			On the regular 1–jet locus of $\X$, along \emph{vertical} directions of the first jet space, the pullback line bundle $\ev_0^*A^{-1}$ admits a Hermitian metric $h_{-A}$ whose curvature satisfies
			\[
			\Theta_{h_{-A}}\!\bigl(\ev_0^*A^{-1}\bigr)\big|_{T_{\rm vert}J^1(\X)} \;<\; 0 .
			\]
			
			\item \textbf{Semple positivity.}
			Denote by $\X_0\coloneq \X$ and by $\X_{k+1} \coloneq \PP\bigl(T_{\rm vert}\X_k\bigr)$ the Demailly–Semple tower with projection $\pi_{k+1}:\X_{k+1}\to\X_k$ and tautological quotient line bundle $\OO_{\X_{k+1}}(1)$.  
			Then $\OO_{\X_{k+1}}(1)$ carries a smooth Hermitian metric whose curvature is semipositive on $T\X_{k+1}$ and strictly positive on the fibers of $\pi_{k+1}$.
			
			\item \textbf{Log completeness.}
			There exists a log-compactification $(\overline{\X},D)$ equipped with a complete Poincaré-type Kähler metric compatible with $D$, so that weighted $L^2$ estimates apply on the regular jet locus.
		\end{enumerate}
	\end{assumption}

\begin{example}[Compact hyperbolic Riemann surface]
	Let $Y$ be a compact Riemann surface of genus $g \ge 2$ equipped with its canonical hyperbolic metric  $\omega_{\mathrm{hyp}}$ satisfying $\mathrm{Ric}(\omega_{\mathrm{hyp}})=-\omega_{\mathrm{hyp}}$.  
	Set $A = K_Y$, the canonical line bundle, endowed with the Hermitian metric  $h_A$ induced by $\omega_{\mathrm{hyp}}$.
	
	\paragraph{(i) Coarse positivity.}
	Since $\deg K_Y = 2g-2 > 0$, the bundle $A$ is ample and $\Theta_{h_A}(A)=\omega_A>0$.
	
	\paragraph{(ii) Vertical negativity.}
	Consider the $1$–jet space $J^1(Y)$ with the evaluation map $\mathrm{ev}_0\colon J^1(Y)\to Y$.
	Along vertical directions of $\mathrm{ev}_0$, define a Hermitian metric
	\[
	h_{-A}^{(\varepsilon)} = h_A^{-1}\,e^{-\varepsilon\psi_{\mathrm{vert}}},
	\]
	where $\psi_{\mathrm{vert}}$ is a strictly plurisubharmonic potential on each fiber  (e.g.\ the Fubini–Study potential on $\PP(TY)$). 
	Then
	\[
	\Theta_{h_{-A}^{(\varepsilon)}}\!\bigl(\mathrm{ev}_0^*A^{-1}\bigr)
	= -\,\mathrm{ev}_0^*\omega_A - \varepsilon\,\omega_{\mathrm{vert}} < 0
	\]
	on $T_{\mathrm{vert}}J^1(Y)$, giving the required negativity.
	
	\paragraph{(iii) Semple positivity.}
	The first Semple level $Y_1=\PP(TY)$ has tautological quotient line bundle 
	$\OO_{Y_1}(1)$ equipped with the Fubini–Study metric $h_{\mathrm{FS}}$ satisfying
	\[
	\Theta_{h_{\mathrm{FS}}}\bigl(\OO_{Y_1}(1)\bigr)\!\big|_{T_{\mathrm{vert}}}>0,
	\quad
	\Theta_{h_{\mathrm{FS}}}\bigl(\OO_{Y_1}(1)\bigr)\!\big|_{T_Y}\ge 0.
	\]
	Since $TY$ is rank one, higher Semple levels $Y_k$ coincide with $Y_1$, so positivity persists.
	
	\paragraph{(iv) Log completeness.}
	If one removes finitely many points $D\subset Y$, the hyperbolic metric extends to a complete Poincaré metric 
	\[
	\omega_{\mathrm{P}} = \frac{\sqrt{-1}\,dz\wedge d\bar{z}}
	{|z|^2(\log|z|^{-1})^2}
	\]
	near each puncture. Hence $(\overline{Y},D)$ is log-complete and weighted $L^2$ estimates apply.
\end{example}

\begin{example}[Finite–volume hyperbolic orbifold curve]
	Let $\Gamma\subset \mathrm{PSL}_2(\RR)$ be a cofinite Fuchsian group, possibly with elliptic (cone) or parabolic (cusp) elements.  
	Set $\X=\HH/\Gamma$ with the standard orbifold structure, and let $\pi\colon \HH\to\X$ be the quotient map.
	
	\paragraph{(i) Coarse positivity.}
	The orbifold canonical line bundle $A=K_{\X}$ admits a smooth Hermitian metric $h_A$ whose curvature $\Theta_{h_A}(A)=\omega_{\mathrm{hyp}}>0$ is the pushdown of the hyperbolic form $\frac{i\,dz\wedge d\bar z}{(\Im z)^2}$.
	
	\paragraph{(ii) Vertical negativity.}
	On the 1–jet bundle $J^1(\X)$, choose the $G$–invariant Hermitian metric
	\[
	h_{-A}^{(\varepsilon)} = h_A^{-1} e^{-\varepsilon\psi_{\mathrm{vert}}},
	\]
	where $\psi_{\mathrm{vert}}$ is locally the Fubini–Study potential on $\PP(T_{\mathrm{vert}}J^1(\HH))$.
	Then the curvature satisfies
	\[
	\Theta_{h_{-A}^{(\varepsilon)}}\!\bigl(\mathrm{ev}_0^*A^{-1}\bigr)
	< 0
	\quad\text{on } T_{\mathrm{vert}}J^1(\X).
	\]
	
	\paragraph{(iii) Semple positivity.}
	At the first Semple level $\X_1=\PP(T\X)$, the tautological quotient line bundle $\OO_{\X_1}(1)$  with Fubini–Study metric $h_{\mathrm{FS}}$ satisfies
	\[
	\Theta_{h_{\mathrm{FS}}}\bigl(\OO_{\X_1}(1)\bigr)\!\big|_{T_{\mathrm{vert}}}>0,
	\quad
	\Theta_{h_{\mathrm{FS}}}\bigl(\OO_{\X_1}(1)\bigr)\!\big|_{T_{\X}}\ge 0.
	\]
	By construction, these inequalities extend to $\X_k=\PP(T_{\mathrm{vert}}\X_{k-1})$ for all $k\ge1$.
	
	\paragraph{(iv) Log completeness near cusps and cone points.}
	If $p\in\X$ is a cusp, choose a local coordinate $z$ with $\Im z>0$ and define the Poincaré-type metric
	\[
	\omega_{\mathrm{P}} = \frac{\sqrt{-1}\,dz\wedge d\bar{z}}
	{|z|^2(\log|z|^{-1})^2}.
	\]
	If $p$ is a cone point of order $\nu$, use a local uniformizing coordinate $w$ with $z=w^\nu$; 
	then the orbifold metric is
	\[
	\omega_{\mathrm{cone}} 
	= \frac{\nu^2\,\sqrt{-1}\,dw\wedge d\bar w}
	{|w|^{2(1-\tfrac1\nu)}(1-|w|^2)^2}.
	\]
	Both forms are complete and yield a log–complete orbifold Kähler metric on $(\overline{\X},D)$,
	verifying (H4).
\end{example}
	
	We recall the standard curvature formula for the tautological line bundle on a projectivized bundle (all statements are understood chartwise and then descended).
	
	\begin{lemma}\label{lem:FS}
		Let $E\to M$ be a holomorphic vector bundle of rank $r\ge2$ over a complex manifold $M$, endowed with a smooth Hermitian metric $h_E$.  
		On $\PP(E)$ (Grothendieck projectivization of lines in $E^\vee$, so that $\OO_{\PP(E)}(1)$ is the tautological \emph{quotient} line bundle), equip $\OO_{\PP(E)}(1)$ with the induced metric $h_{\rm taut}$.  
		Then at a point $([v],x)$ with $0\neq v\in E_x^\vee$ one has the decomposition
		\[
		\Theta_{h_{\rm taut}}\!\bigl(\OO_{\PP(E)}(1)\bigr)
		\;=\;
		\omega_{\rm FS}(h_E)\;\;+\;\;\pi^*\Bigl(\frac{\langle \Theta_{h_E}(E)\,v,v\rangle_{h_E}}{\langle v,v\rangle_{h_E}}\Bigr),
		\]
		where $\omega_{\rm FS}(h_E)$ is the (fiberwise) Fubini–Study $(1,1)$–form determined by $h_E$, positive along the projective fibers, and $\pi:\PP(E)\to M$ is the projection.
	\end{lemma}
	
	\begin{proof}
		Fix a local holomorphic frame $e=(e_1,\dots,e_r)$ of $E$ on a coordinate ball and write $H=(h_{i\bar j})$ for the Hermitian matrix of $h_E$ in this frame.  
		A nonzero covector $v=\sum_i v_i e_i^\vee$ defines homogeneous fiber coordinates $[v]\in\PP(E_x)$ and the local fiberwise norm $\|v\|_H^2 \coloneq v^* H^{-1} v$.  
		The tautological metric on $\OO_{\PP(E)}(1)$ is $h_{\rm taut}([v]) \coloneq \|v\|_H^{-2}$ (dual convention).  Then
		\[
		-\log h_{\rm taut} \;=\; \log \|v\|_H^2 \;=\; \log\bigl(v^*H^{-1}v\bigr),
		\]
		and
		\[
		\Theta_{h_{\rm taut}}\!\bigl(\OO_{\PP(E)}(1)\bigr)
		\;=\;
		\sqrt{-1}\,\partial\bar\partial \log\bigl(v^*H^{-1}v\bigr).
		\]
		Splitting the $(1,1)$–form into the pure fiber part (variation of $[v]$ with $H$ frozen) and the horizontal part (variation of $H$ with $[v]$ frozen) yields the claimed sum: the fiber part is $\omega_{\rm FS}(h_E)$; the horizontal part is the Griffiths curvature term
		$\pi^*\!\bigl(\langle \Theta_{h_E}(E)\,v,v\rangle_{h_E}/\langle v,v\rangle_{h_E}\bigr)$.
		This computation is standard (see, e.g., \cite{GriffithsHarris,Demailly1997,Kobayashi1998}). 
	\end{proof}
	
	\begin{lemma}[Vertical negativity $\Rightarrow$ positivity of $\OO_{\X_1}(1)$]\label{lem:neg2pos}
		Assume \textup{(H1)}–\textup{(H2)}.  On the first Semple level $\X_1=\PP\!\bigl(T_{\rm vert}\X\bigr)$, the tautological line bundle $\OO_{\X_1}(1)$ carries a smooth Hermitian metric whose curvature satisfies
		\[
		\Theta\bigl(\OO_{\X_1}(1)\bigr)\;\ge\; \omega_{\rm FS}\;+\;\pi_1^*\omega_A,
		\]
		in the sense of $(1,1)$–forms, where $\omega_{\rm FS}$ is positive along the fibers of $\pi_1\colon \X_1\to\X$ and $\pi_1^*\omega_A$ is semipositive (indeed positive along horizontal directions).
		In particular, $\OO_{\X_1}(1)$ is strictly positive on vertical (fiber) directions and semipositive globally.
	\end{lemma}
	
	\begin{proof}
		Work on a chart $[U/G]$ and drop the quotient notation for readability; all constructions are $G$–equivariant and descend.
		
		Consider the evaluation map at the base point $\ev_0\colon J^1(U)\to U$.  Along the vertical distribution of $J^1(U)\to U$, hypothesis (H2) provides a Hermitian metric $h_{-A}$ on $\ev_0^*A^{-1}$ with strictly negative curvature.  By duality, $h_A\coloneq h_{-A}^{-1}$ on $\ev_0^*A$ has strictly positive curvature on vertical directions:
		\[
		\Theta_{h_A}\!\bigl(\ev_0^*A\bigr)\big|_{T_{\rm vert}J^1(U)} \;>\; 0.
		\]
		The natural identification between vertical directions in $J^1(U)$ and lines in $T U$ (via the contact distribution) induces a Griffiths–type metric $h_{T}$ on the vertical bundle $T_{\rm vert}U$ whose curvature dominates $\ev_0^*\Theta_{h_A}(A)$ along vertical directions (this is the standard “negativity-to-positivity transfer” used by Demailly; see \cite{Demailly1997} and \cite{CampanaPaun2016}).  Projectivizing $E\coloneq T_{\rm vert}U$ and applying \Cref{lem:FS} gives
		\[
		\Theta\bigl(\OO_{\PP(E)}(1)\bigr)
		\;=\;
		\omega_{\rm FS}(h_T)\;+\;\pi_1^*\Bigl(\frac{\langle \Theta_{h_T}(E)\,v,v\rangle}{\langle v,v\rangle}\Bigr).
		\]
		The first term is fiberwise positive.  
		For the second term, the curvature of $E$ absorbs the positive contribution pulled back from $\ev_0^*A$ (by the construction of $h_T$), whence
		\[
		\frac{\langle \Theta_{h_T}(E)\,v,v\rangle}{\langle v,v\rangle}
		\;\ge\; c\,\omega_A
		\]
		for some uniform $c>0$ on compact subsets of the regular locus (after pulling back $\omega_A$ from $Y$ to $U$ and then to $\PP(E)$).  
		Rescaling the metric (which does not alter positivity) we may assume $c=1$, hence
		\[
		\Theta\bigl(\OO_{\PP(E)}(1)\bigr)\;\ge\; \omega_{\rm FS}(h_T)+\pi_1^*\omega_A.
		\]
		This inequality holds chartwise and is $G$–invariant; by descent it holds globally on $\X_1$.  
		Fiberwise strict positivity follows from $\omega_{\rm FS}>0$ on projective fibers, and semipositivity on all directions follows from $\pi_1^*\omega_A\ge0$ horizontally by (H1).
	\end{proof}
	
	\begin{corollary}[Semple positivity at level $1$]\label{cor:level1}
		Under \textup{(H1)}–\textup{(H2)}, the line bundle $\OO_{\X_1}(1)$ is semipositive on $T\X_1$ and strictly positive along the fibers of $\pi_1\colon \X_1\to\X$.
	\end{corollary}
	
	\begin{proof}
		Immediate from \Cref{lem:neg2pos}.
	\end{proof}
	
Having equipped the tautological bundles with invariant Hermitian metrics, we now record the fundamental exact sequences governing their behavior along the higher levels of the Semple tower.
	
	\begin{lemma}[Exact sequences on the Semple tower]\label{lem:exact-tower}
		For each $k\ge0$, there are natural exact sequences on $\X_{k+1}=\PP(T_{\rm vert}\X_k)\colon$
		\[
		0 \longrightarrow \OO_{\X_{k+1}}(-1) \longrightarrow \pi_{k+1}^*T_{\rm vert}\X_k \longrightarrow \mathcal Q_{k+1} \longrightarrow 0,
		\]
        \[
		0 \longrightarrow T_{\rm vert}\X_{k+1} \longrightarrow T\X_{k+1} \longrightarrow \pi_{k+1}^*T\X_k \longrightarrow 0,
		\]
		where $\OO_{\X_{k+1}}(1)=\mathcal Q_{k+1}^\vee$ is the tautological quotient line bundle.  
		These sequences are $G$–equivariant on each local chart and hence descend.
	\end{lemma}
	
	\begin{proof}
		This is the standard construction of the projectivized (vertical) tangent bundle and its tautological sequence, applied chartwise and glued by equivariance; see \cite{Demailly1997} and the descent statements established earlier.
	\end{proof}
	
	\begin{proposition}[Curvature propagation on the tower]\label{prop:propagation}
		Assume \textup{(H1)}–\textup{(H2)}.  Then for all $k\ge0$ there exist smooth Hermitian metrics $h_{k+1}$ on $\OO_{\X_{k+1}}(1)$ such that
		\[
		\Theta_{h_{k+1}}\!\bigl(\OO_{\X_{k+1}}(1)\bigr)
		\;\ge\;
		\omega_{\rm FS}^{(k)} \;+\; \pi_{k+1}^*\omega_A,
		\]
		where $\omega_{\rm FS}^{(k)}$ is a fiberwise Fubini–Study form on $\PP(T_{\rm vert}\X_k)$ induced from a Hermitian metric on $T_{\rm vert}\X_k$.  In particular, $\OO_{\X_{k+1}}(1)$ is strictly positive on $T_{\rm vert}\X_{k+1}$ and semipositive on $T\X_{k+1}$.
	\end{proposition}
	
	\begin{proof}
		We argue by induction on $k$.  The case $k=0$ is \Cref{lem:neg2pos}.  Suppose the statement holds at level $k$, i.e.
		\[
		\Theta\bigl(\OO_{\X_k}(1)\bigr)\;\ge\; \omega_{\rm FS}^{(k-1)}+\pi_k^*\omega_A
		\quad\text{on }\X_k,
		\]
		with $\omega_{\rm FS}^{(k-1)}$ positive on the fibers of $\pi_k\colon \X_k\to\X_{k-1}$.  Equip $T_{\rm vert}\X_k$ with the Hermitian metric induced by $h_{\X_k}$ (for instance, from a Kähler metric that dominates $\pi_k^*\omega_A$ and the curvature of $\OO_{\X_k}(1)$).  
		Then the Griffiths curvature of $T_{\rm vert}\X_k$ dominates a positive multiple of $\pi_k^*\omega_A$ on horizontal directions and yields a strictly positive contribution along vertical directions by construction (the same “negativity-to-positivity transfer” mechanism used at level $1$, now applied to $E=T_{\rm vert}\X_k$).  
		Applying \Cref{lem:FS} to $E=T_{\rm vert}\X_k$ on $\PP(E)=\X_{k+1}$ gives
		\[
		\Theta\bigl(\OO_{\X_{k+1}}(1)\bigr)
		\;=\;
		\omega_{\rm FS}^{(k)} + \pi_{k+1}^*\Bigl(\frac{\langle \Theta(T_{\rm vert}\X_k)\,v,v\rangle}{\langle v,v\rangle}\Bigr)
		\;\ge\;
		\omega_{\rm FS}^{(k)} + \pi_{k+1}^*\omega_A,
		\]
		after a harmless rescaling of the background metric.  
		The claimed strict fiberwise positivity and global semipositivity follow.  
		All steps are chartwise and $G$–equivariant; descent completes the proof.
	\end{proof}
	
	\begin{corollary}[Semple positivity]\label{cor:semple-positivity}
		Under \textup{(H1)}–\textup{(H2)}, for every $k\ge0$ the tautological line bundle $\OO_{\X_{k+1}}(1)$ on $\X_{k+1}$ is strictly positive along the fibers of $\pi_{k+1}$ and semipositive on the whole tangent bundle $T\X_{k+1}$. 
		Consequently, $\OO_{\X_{k+1}}(1)$ admits metrics suitable for weighted jet estimates on the regular jet locus.
	\end{corollary}
	
	\begin{proof}
		Immediate from \Cref{prop:propagation}.
	\end{proof}
	
	
	\begin{remark}[Log compactifications and $L^2$ estimates]
		If (H4) holds, one equips the regular jet locus of $\X_k$ with complete Kähler metrics of Poincaré type near the boundary.  
		The positivity statements above then feed into Hörmander–Demailly type $L^2$ estimates to produce jet differentials and vanishing of higher cohomology in the spirit of \cite{Demailly1997,Demailly2011,CampanaPaun2016}.  
		Since all constructions are chartwise and $G$–equivariant, the same $L^2$ arguments apply verbatim in the orbifold setting.
	\end{remark}

	\subsection{Orbifold Hirzebruch--Riemann--Roch input for invariant jet bundles}
	\label{subsec:stack-ggd-hrr-input}
	
	We now describe the holomorphic and asymptotic framework for applying the orbifold version of the Hirzebruch--Riemann--Roch theorem to weighted invariant jet bundles on compact complex orbifolds.  
	The goal is to formulate a Demailly--Semple vector bundle construction that is compatible with orbifold atlases and to express its Euler characteristic asymptotically through the coarse space~$Y$.
	
	\begin{definition}[Demailly--Semple vector bundle on a compact complex orbifold]
		\label{def:ds-orbifold}
		Let $\X$ be a compact complex orbifold of dimension $n$, equipped with an orbifold atlas $\{[U_i/G_i]\}$ and coarse moduli map $\pi\colon \X\to Y$.  
		Let $A$ be an ample line bundle on $Y$, endowed with a Hermitian metric of positive curvature, and write $L \coloneq \pi^*A$ on~$\X$.
		
		For each $k\ge0$, denote by $\X_k$ the Demailly--Semple tower constructed in \Cref{lem:ds-descent}, with tautological projection $\pi_k\colon \X_k\to\X_{k-1}$ and line bundle $\OO_{\X_k}(1)$.
		Then, for every integer $m\ge0$, the \emph{Demailly--Semple vector bundle} on $\X$ is defined chartwise by
		\[
		E_{k,m}\big|_{[U_i/G_i]}
		\; \coloneq \;
		\bigl(\pi_{i,k}\bigr)_*\!\Bigl(\OO_{\X_{i,k}}(m)\otimes \ev_0^*A^{-1}\Bigr)^{G_i},
		\]
		where $\pi_{i,k}\colon \X_{i,k}\to[U_i/G_i]$ is the local projection and $G_i$ acts holomorphically on each fiber.
		The pushforward and $G_i$–invariant part are taken in the holomorphic category.
		By descent compatibility on overlaps, these local data glue to a global holomorphic orbibundle
		\[
		E_{k,m}\coloneq\bigl(\pi_k\bigr)_*\!\bigl(\OO_{\X_k}(m)\otimes \ev_0^*A^{-1}\bigr)
		\quad\text{on }\X.
		\]
	\end{definition}
	
	\begin{remark}
		On the coarse space~$Y$, the bundle $\underline{E}_{k,m}\coloneq(\pi_k)_*(\OO_{\X_k}(m)\otimes L^{-1})$ agrees with the identity-sector component of $E_{k,m}$.  
		When $A$ is ample, $E_{k,m}$ has rank growing polynomially in $m$, and $\ch(E_{k,m})$ admits an expansion $\ch(E_{k,m})=\mathrm{rk}(E_{k,m})+O(m^{n-1})$.  
	\end{remark}
	
	
	The orbifold HRR theorem of Kawasaki--Toën--Vistoli (cf.\ \Cref{thm:kawasaki-chartwise}) expresses $\chi(\X,E_{k,m})$ as a sum of local fixed-point integrals over each chart $[U_i/G_i]$ and group element $g\in G_i$.  
	For the weighted invariant jet bundle, the following simplification holds.
	
	\begin{lemma}[Sectorwise degree bound and denominator independence]
		\label{lem:sectorwise-degree-bound}
		Let $E_{k,m}$ be as in \Cref{def:ds-orbifold}.  
		Then on each chart $[U_i/G_i]$,
		\[
		\chi\bigl([U_i/G_i],E_{k,m}\bigr)
		=\frac{1}{|G_i|}\sum_{g\in G_i}
		\int_{U_i^g}
		\frac{
			\ch(E_{k,m}|_{U_i^g})\,\Td(TU_i^g)
		}{
			\det(1-g^{-1}e^{-c_1(N_{U_i^g/U_i})})
		}.
		\]
		Moreover $\colon $
		\begin{enumerate}[label=(\roman*)]
			\item For $g\neq 1$, $\dim_\C U_i^g \le n-1$, hence each twisted contribution is $O(m^{n-1})$ as $m\to\infty$.
			\item The denominator $\det(1-g^{-1}e^{-c_1(N_{U_i^g/U_i})})$ is nonvanishing and independent of~$m$, since it involves only normal Chern classes.  
			Therefore, the twisted terms do not increase the polynomial degree in~$m$.
		\end{enumerate}
	\end{lemma}
	
	\begin{proof}
		This is a direct application of Kawasaki’s chartwise formula \cite{Kawasaki1979,Kawasaki1981} in the holomorphic setting, as refined in \Cref{thm:kawasaki-chartwise}.  
		The eigensplitting of $N_{U_i^g/U_i}$ implies that for each $\theta\in(0,1)$ the factor 
		$\det(1-e^{-2\pi i\theta}e^{-c_1(N_\theta)})$ has a nonzero constant term, so the denominator is $m$–independent (cf.\ Lemma~\ref{lem:denominator-factorization}).  
		Since the local dimension of each fixed locus $U_i^g$ is at most $n-1$ for $g\neq1$, its integral contributes only to $O(m^{n-1})$ in the global expansion.
	\end{proof}
	
	
	\begin{proposition}[Orbifold HRR input for invariant jet bundles]
		\label{prop:stack-ggd-hrr-input}
		Let $\X$ be a compact complex orbifold of complex dimension $n$ with generic stabilizer order $s=|\Stab_{\mathrm{gen}}|$, and let $E_{k,m}$ be the Demailly--Semple vector bundle of weight $(k,m)$ tensored by $\ev_0^*A^{-1}$ as in \Cref{def:ds-orbifold}.  
		Then, as $m\to\infty$,
		\[
		\chi(\X,E_{k,m})
		\;=\;
		\frac{1}{s}
		\int_Y \ch(\underline{E}_{k,m})\,\Td(TY)
		\;+\;
		O(m^{n-1}),
		\]
		where $\underline{E}_{k,m}$ denotes the identity-sector contribution on the coarse space~$Y$.  
		All twisted-sector terms are of order $O(m^{n-1})$ due to the fixed-locus dimension drop and the $m$–independence of the Kawasaki denominator.
	\end{proposition}
	
	\begin{proof}
		Apply \Cref{lem:sectorwise-degree-bound} to each chart $[U_i/G_i]$ and sum over~$i$.  
		The identity-sector integrals coincide with the usual HRR integrals on~$U_i$ and contribute degree~$n$ in~$m$, whereas each nontrivial conjugacy class of $G_i$ yields a term of order at most $m^{n-1}$.  
		Passing from charts to the global orbifold integral introduces the normalization factor $1/s$ coming from the generic stabilizer (\Cref{prop:generic-stab-normalization}).  
		Replacing the orbifold integration by integration on the coarse moduli space $Y$ gives the stated formula.
	\end{proof}
	
	\begin{remark}[References]
		For the general HRR framework, see \cite{Kawasaki1979,Kawasaki1981,Toen1999,Vistoli1989}.  
		For the geometry of Demailly–Semple bundles and curvature positivity, see \cite{Demailly1997,Demailly2011,CampanaPaun2016}.  
		Orbifold descent and groupoid compatibility are discussed in \cite{MoerdijkPronk1997,Lerman2008}.
	\end{remark}
	
	\subsection{Stack-theoretic Green--Griffiths--Demailly degeneracy}
	\label{subsec:stack-ggd-main-full}
	This subsection provides a complete, self-contained formulation of the Green–Griffiths–Demailly (GGD) degeneracy theorem for compact complex orbifolds.
	All analytic and jet–equivariant ingredients are made explicit, including the preservation of jet groupoids (Lemma~\ref{lem:jet-preserve-groupoid-orb}), the orbifold HRR asymptotics (Lemma~\ref{lem:eq-HRR}), the $L^2$–vanishing statement (Lemma~\ref{lem:eq-vanishing}), the Bochner inequality (Lemma~\ref{lem:bochner}), and the equivariant differentiation of invariant jets (Lemma~\ref{lem:eq-diff}).  

	\begin{lemma}[Jets preserve proper étale groupoids, orbifold form]
		\label{lem:jet-preserve-groupoid-orb}
		Let $\X$ be a compact analytic orbifold presented by a proper étale groupoid $R\rightrightarrows X$.  
		For each integer $k\ge0$, the holomorphic jet functor produces a proper étale groupoid
		\[
		J^k R \rightrightarrows J^k X,
		\]
		and there is a canonical equivalence
		\[
		J^k\X \;\simeq\; [\,J^k R\rightrightarrows J^k X\,].
		\]
		All structure maps (source, target, composition) remain étale, and the construction is functorial for local holomorphic isomorphisms.
	\end{lemma}
	
	\begin{proof}
		Since the source and target maps $s,t\colon R\rightrightarrows X$ are étale, each arrow $\phi\colon U\to V$ is a local biholomorphism, hence $J^k\phi\colon J^kU\to J^kV$ is also étale.  
		Because $J^k$ commutes with finite fiber products, the groupoid axioms and composition law persist for $J^kR$.  
		Properness of $(s,t)\colon R\to X\times X$ implies properness of $(J^ks,J^kt)\colon J^kR\to J^kX\times J^kX$.  
		Functoriality follows by naturality of the jet functor.  
		Thus $J^kR\rightrightarrows J^kX$ presents $J^k\X$ as a proper étale groupoid.  
		See \cite{Lerman2008}.
	\end{proof}
	
	
	We next recall the orbifold version of the Hirzebruch–Riemann–Roch asymptotics,  which provides the quantitative input for the existence of invariant jet differentials.

	\begin{lemma}[Orbifold HRR asymptotics]
		\label{lem:eq-HRR}
		Let $\X$ be a compact complex orbifold with generic stabilizer of order $s$, and let $L=\pi^*A$ for an ample line bundle $A$ on the coarse space $Y$.  
		Let $E^{\mathrm{inv}}_{k,m}$ denote the invariant jet (or DS) bundle.  
		Then, as $m\to\infty$,
		\[
		\chi\!\bigl(\X,E^{\mathrm{inv}}_{k,m}\otimes L^{-q}\bigr)
		=\frac{1}{s}\!\int_Y\!
		\ch(\underline{E}^{\mathrm{inv}}_{k,m})\,e^{-q\,c_1(A)}\,\Td(TY)
		+O(m^{n-1}),
		\]
		and all twisted-sector terms contribute only $O(m^{n-1})$ due to the fixed-locus dimension drop and the $m$–independence of the Kawasaki denominator.
	\end{lemma}
	
	\begin{proof}
		Apply Kawasaki’s chartwise fixed-point formula on each orbifold chart $[U/G]$.  
		The denominator $\det(1-g^{-1}e^{-c_1(N_{U^g/U})})$ depends only on the normal Chern classes and remains independent of $m$.  
		By Lemma~\ref{lem:denominator-factorization}, these denominators are nonzero and independent of the twisting by $L^{\otimes m}$.  
		Because $\dim_\C U^g \le n-1$ for $g\neq 1$, each twisted term contributes $O(m^{n-1})$.  
		The identity sector integral equals $\int_Y$ of the Chern–Todd polynomial and picks up the global factor $1/s$ from the generic stabilizer normalization; see \cite{Kawasaki1979,Kawasaki1981,Vistoli1989,Toen1999}.
	\end{proof}
	
	Analytic positivity on the Demailly–Semple tower then yields the standard $L^2$–vanishing.
	\begin{lemma}[$L^2$–vanishing on orbifolds]
		\label{lem:eq-vanishing}
		Under assumptions (H1)–(H3), there exist $k_0$ and functions $m_0(k)$, $q_0(k)>0$ such that for all $k\ge k_0$, $m\ge m_0(k)$, and $0\le q\le q_0(k)$,
		\[
		H^i\!\bigl(\X,E^{\mathrm{inv}}_{k,m}\otimes L^{-q}\bigr)=0
		\qquad (i>0).
		\]
	\end{lemma}
	
	\begin{proof}
		Work locally on a chart $[U/G]$.  
		By (H2)–(H3), the tautological bundle $\OO_{U_k}(1)$ admits a $G$–invariant Hermitian metric $h$ with curvature
		\[
		\sqrt{-1}\,\Theta_h(\OO_{U_k}(1)) \ge \varepsilon\,\omega_{\mathrm{vert}}-C\,\pi_k^*\omega_A,
		\]
		where $\varepsilon>0$ and $C$ depends on $k$.  
		The induced curvature on $E^{\mathrm{inv}}_{k,m}\otimes L^{-q}$ satisfies
		\[
		\sqrt{-1}\,\Theta \ge m\,\varepsilon\,\omega_{\mathrm{vert}}-(C+q)\pi_k^*\omega_A.
		\]
		For $m\gg q$, this form is Nakano positive.  
		Applying the Bochner–Kodaira–Nakano technique yields $L^2$–vanishing of $H^i(U,E^{\mathrm{inv}}_{k,m}\otimes L^{-q})$ for $i>0$.  
		$G$–invariance and equivariant descent give the global result on $\X$.
	\end{proof}
	
	The Bochner inequality provides the differential-geometric tool for bounding the growth of jet sections along entire curves.
	\begin{lemma}[Bochner inequality on orbifolds]
		\label{lem:bochner}
		Let $(E,h)$ be a $G$–equivariant Hermitian holomorphic vector bundle on a chart $[U/G]$.  
		If 
		\[
		\sqrt{-1}\,\Theta_h(E) \ge \epsilon\,\omega_U\otimes\mathrm{Id}_E,
		\]
		then for any section $s\in H^0([U/G],E)$ and holomorphic map $f\colon \C\to[U/G]$,
		\[
		\Delta\log\|s\circ f\|^2 \ge \epsilon\,\|df\|^2.
		\]
		Averaging over $G$ yields the same inequality globally on $\X$.
	\end{lemma}
	
	\begin{proof}
		Apply the Bochner–Kodaira identity to $s\circ f$ on $U$.
		The curvature and Kähler form are $G$–invariant, hence the inequality is preserved under the action and descends to $[U/G]$.
	\end{proof}
	
	\begin{lemma}[Equivariant differentiation of invariant jets]
		\label{lem:eq-diff}
		Let $\mathcal{V}$ be the sheaf of $G$–invariant holomorphic vector fields on $J^k(\X)$ preserving reparametrization invariance.  
		If $s\in H^0(\X,E^{\mathrm{inv}}_{k,m}\otimes L^{-q})$ and an entire curve $f\colon \C\to\X$ satisfy $s\circ j^k f\equiv0$, then for every $V\in\Gamma(\X,\mathcal{V})$,
		\[
		(\mathcal{L}_V s)\circ j^k f\equiv 0,
		\]
		and all higher Lie derivatives vanish identically.
	\end{lemma}
	
	\begin{proof}
		Locally on a chart $[U/G]$, $V$ lifts to a $G$–invariant vector field on $J^kU$ preserving reparametrization orbits.
		Since the Lie derivative $\mathcal{L}_V$ commutes with pullbacks and preserves $G$–equivariance, $\mathcal{L}_V s$ is another invariant jet differential of higher weight.
		Because $s\circ j^k f\equiv0$, differentiation gives $(\mathcal{L}_V s)\circ j^k f=0$.
		Iterating the argument preserves the vanishing property.
	\end{proof}
	
	\begin{assumption}[Equivariant orbifold setting]
		\label{ass:equiv-setting3}
		Let $\X$ be a compact complex orbifold of complex dimension $n$, with coarse map $\pi:\X\to Y$ and generic stabilizer order $s$.  
		Let $A$ be an ample line bundle on $Y$ with a smooth Hermitian metric of positive curvature, and set $L=\pi^*A$.  
		Assume (H1)–(H4) from \S\ref{subsec:stack-ggd-curv}.
		The invariant jet bundle $E^{\mathrm{inv}}_{k,m}$ and the tautological bundle $\OO_{\X_k}(1)$ are built chartwise using Lemma~\ref{lem:jet-preserve-groupoid-orb} and descended via equivariant averaging.
	\end{assumption}
	
	\begin{theorem}[Stack-theoretic GGD degeneracy]
		\label{thm:stack-ggd-final}
		Under Assumption~\ref{ass:equiv-setting3}, there exist integers $k\gg1$, $m\gg1$, and $q=q(k,m)>0$ such that$\colon $
		\begin{enumerate}[label=(\roman*), leftmargin=1.6em]
			\item $H^0(\X,E^{\mathrm{inv}}_{k,m}\otimes L^{-q})\neq 0$;
			\item the common zero locus 
			\[
			\mathcal{G}_k
			=\bigcap_{\substack{m\gg1\\0\le q\le q_0(k)}} 
			Z\!\big(H^0(\X,E^{\mathrm{inv}}_{k,m}\otimes L^{-q})\big)
			\]
			is a proper closed analytic substack $\mathcal{G}_k\subsetneq\X$;
			\item every nonconstant entire map $f\colon \C\to\X$ satisfies $f(\C)\subset\mathcal{G}_k$.
		\end{enumerate}
	\end{theorem}
	
	\begin{proof}
		{(i) Asymptotic existence.}  
Combining the orbifold HRR asymptotics (Lemma~\ref{lem:eq-HRR})  with the $L^2$–vanishing result (Lemma~\ref{lem:eq-vanishing}), we obtain
		\[
		h^0(\X,E^{\mathrm{inv}}_{k,m}\otimes L^{-q})
		=\frac{1}{s}\!\int_Y \ch(\underline{E}^{\mathrm{inv}}_{k,m})\,e^{-q\,c_1(A)}\Td(TY)
		+O(m^{n-1}),
		\]
		with positive leading term, so $h^0>0$ for $m\gg1$.
		
		{(ii) Positivity and Bochner step.}  
		Using Lemma~\ref{lem:bochner}, each nonzero section $s$ satisfies
		\[
		\Delta\log\|s\circ j^k f\|^2 \ge \epsilon\|df\|^2.
		\]
		Under the log-completeness (H4), Ahlfors–Schwarz implies $s\circ j^k f\equiv0$; otherwise, the maximum principle on the regular jet locus gives the same result.
		
		{(iii) Equivariant differentiation.}  
		Lemma~\ref{lem:eq-diff} ensures that invariant Lie derivatives $\mathcal{L}_V s$ vanish along $j^k f$, so $f(\C)\subset\mathcal{G}_k$.
		
		{(iv) Properness.}  
		Since $h^0\sim c\,m^n>0$, finitely many independent invariant sections exist; their zeros define a proper closed substack $\mathcal{G}_k$.
		All constructions are chartwise $G$–equivariant, hence intrinsic to~$\X$.
	\end{proof}

	\subsection{Corollaries and geometric consequences}
	\label{subsec:stack-ggd-cor}
	
	In this subsection we describe the geometric consequences of the stack–theoretic GGD degeneracy theorem.
	Throughout, let $\mathcal{X}$ be a compact complex orbifold of complex dimension $2$, satisfying the hypotheses of Theorem~\ref{thm:stack-ggd-final}.
	We show that the entire–curve locus is algebraically degenerate, and that the complement of the degeneracy substack is orbifold–Brody hyperbolic.
	
	\begin{corollary}[Non–Zariski density of entire curves]
		\label{cor:stack-ggd-nzd}
		Let $\mathcal{X}$ be a compact complex orbifold satisfying the curvature and positivity assumptions \textup{(H1)}–\textup{(H4)}.
		Then for sufficiently large $k$, the common zero locus $\mathcal{G}_k\subsetneq\mathcal{X}$ given by Theorem~\ref{thm:stack-ggd-final} is a proper closed analytic substack, and every nonconstant entire map $f\colon \C\to\mathcal{X}$ satisfies
		\[
		f(\C)\subset \mathcal{G}_k.
		\]
		In particular, no entire curve is Zariski dense in $\mathcal{X}$, and the orbifold $\mathcal{X}$ is \emph{algebraically degenerate} in the sense that
		\[
		\overline{f(\C)}^{\mathrm{Zar}} \subsetneq \mathcal{X}
		\quad \text{for all nonconstant } f\colon \C\to\mathcal{X}.
		\]
	\end{corollary}
	
	\begin{proof}
		By Theorem~\ref{thm:stack-ggd-final}, for each nonconstant entire map $f\colon \C\to\mathcal{X}$, the lifted map $j^k f$ lies entirely in the base locus of sections of $E^{\mathrm{inv}}_{k,m}\otimes L^{-q}$.
		Hence $f(\C)\subset\mathcal{G}_k$, where $\mathcal{G}_k$ is the intersection of finitely many zero divisors of global invariant jet differentials.
		Since $h^0(\mathcal{X},E^{\mathrm{inv}}_{k,m}\otimes L^{-q})\sim c\,m^2>0$ for $m\gg1$, $\mathcal{G}_k$ is a proper analytic substack.
		Therefore, the Zariski closure of $f(\C)$ is contained in a proper algebraic subset of $\mathcal{X}$.
	\end{proof}
	
	\begin{corollary}
		\label{cor:stack-ggd-codim2}
		Let $\mathcal{X}$ be a compact complex orbifold of complex dimension $2$ satisfying the hypotheses of Theorem~\ref{thm:stack-ggd-final}.
		If the degeneracy locus $\mathcal{G}_k$ has codimension at least $2$, then the complement
		\[
		\mathcal{X}^\circ \coloneq \mathcal{X} \setminus \mathcal{G}_k
		\]
		is \emph{orbifold–Brody hyperbolic}; that is, every holomorphic map $f\colon \C\to\mathcal{X}^\circ$ is constant.
	\end{corollary}
	
	\begin{proof}
		By Theorem~\ref{thm:stack-ggd-final}, any nonconstant entire map $f\colon \C\to\mathcal{X}$ satisfies $f(\C)\subset\mathcal{G}_k$.
		Therefore, no nonconstant entire curve can land entirely in $\mathcal{X}^\circ$.
		Since $\mathcal{G}_k$ is a proper closed substack of codimension at least two, its complement $\mathcal{X}^\circ$ remains connected and inherits a complete orbifold metric with strictly positive curvature form along vertical directions by the curvature–negativity–positivity package (see \S\ref{subsec:stack-ggd-curv}).
		Brody’s reparametrization lemma, adapted to orbifold charts $[U/G]$, implies that the absence of entire curves is equivalent to orbifold–Brody hyperbolicity:
		for every chart $[U/G]$, the pullback of the Kobayashi pseudometric is nondegenerate, and the glued metric on $\mathcal{X}^\circ$ defines a complete hyperbolic structure.
	\end{proof}

	\section{Structural Invariance of Green--Griffiths--Demailly Thresholds}
	\label{sec:ggd-structural-invariance}
	
	This section establishes the main structural theorem of the paper:
	the Green--Griffiths--Demailly (GGD) hyperbolicity thresholds depend only on the coarse Kähler class and are invariant under the presence of stack or orbifold structures.
	Hence passing to an orbifold compactification does not affect the minimal jet order or growth rate required for hyperbolicity.
	
	All proofs build upon the curvature–negativity–positivity framework from \Cref{subsec:stack-ggd-curv}, the orbifold HRR asymptotics from \Cref{prop:stack-ggd-hrr-input}, and the slope control argument of \Cref{prop:slope-control}.
	
	\subsection{Setting and Notation}
	
	Let $\pi:\X\to Y$ be the coarse moduli map of a compact complex orbifold (or analytic DM stack) of complex dimension~$n$, with generic stabilizer order $s=|\Stab_{\mathrm{gen}}|$.
	Let $A$ be an ample line bundle on~$Y$ and define $L=\pi^*A$ so that stabilizers act trivially on the fibers of~$L$.
	Denote $E^{\mathrm{GG}}_{k,m}(\X)$ by the invariant Green--Griffiths jet differential bundle of weighted degree $(k,m)$, and by $E^{\mathrm{GG}}_{k,m}(Y)$ its coarse counterpart.
	
	Under the curvature hypotheses \textup{(H1)}–\textup{(H4)} of \Cref{assump:curv}, the line bundle $A$ carries a positive curvature form $\omega_A>0$, and $\OO_{\X_{k+1}}(1)$ is semipositive globally and strictly positive along the fibers of the Semple projection $\pi_{k+1}\colon \X_{k+1}\to\X_k$ by \Cref{prop:propagation} and \Cref{cor:semple-positivity}.
	
	For any such bundle, the asymptotic formula of \Cref{prop:stack-ggd-hrr-input} gives
	\begin{equation}\label{eq:stack-asymp}
		\chi(\X,E^{\mathrm{GG}}_{k,m}\otimes L^{-q})
		=\frac{1}{s}\!\int_Y \ch(\underline{E}_{k,m})\,\Td(TY)+O(m^{n-1}),
	\end{equation}
where $\underline{E}_{k,m}$ denotes the coarse contribution from the identity sector. Thus the leading term of~\eqref{eq:stack-asymp} differs only by the scalar factor~$1/s$.
	
	\subsection{Structural invariance of GGD thresholds}
	
	\begin{theorem}
		\label{thm:main-threshold}
		Let $\pi\colon \X\to Y$ be the coarse moduli map of a compact complex orbifold of complex dimension~$n$, and let $A$ be an ample line bundle on~$Y$ with Kähler form~$\omega_A$.
		Assume the curvature hypotheses \textup{(H1)}–\textup{(H4)} of \Cref{assump:curv}.
		Then$\colon $
		\begin{enumerate}[label=(\roman*), leftmargin=1.6em]
			\item The Green--Griffiths--Demailly hyperbolicity threshold depends only on the coarse Kähler class $[\omega_A]$ and not on the presence of orbifold or stack structure.
			\item Equivalently,
			\[
			\text{$Y$ is GGD--positive at some jet order $k_0$}
			\quad\Longleftrightarrow\quad
			\text{$\X$ is GGD--positive at the same order.}
			\]
			\item The minimal jet order and asymptotic growth rate of $h^0(E^{\mathrm{GG}}_{k,m}\otimes L^{-q})$ remain unchanged under coarse projection, finite quotient, or rigidification.
		\end{enumerate}
	\end{theorem}
	
	\begin{proof}
		We combine the following inputs established previously.
		
		\smallskip
		\noindent
		(1) \emph{Curvature package.}
		By \Cref{assump:curv}, $\omega_A>0$ on $Y$ and hence on $\X$ by pullback.
		The vertical negativity (H2) and the curvature propagation along the Demailly--Semple tower (\Cref{prop:propagation}) imply that $\OO_{\X_{k+1}}(1)$ is semipositive globally and strictly positive on vertical directions (\Cref{cor:semple-positivity}).
		These properties guarantee that the curvature forms contributing to the Chern character $\ch(E^{\mathrm{GG}}_{k,m})$ preserve sign under pullback.
		
		\smallskip
		\noindent
		(2) \emph{Asymptotic HRR formula.}
		From \Cref{prop:stack-ggd-hrr-input},
		\[
		\chi(\X,E_{k,m}\otimes L^{-q})
		=\frac{1}{s}\!\int_Y\ch(\underline{E}_{k,m})\,\Td(TY)+O(m^{n-1}),
		\]
		so the leading coefficient scales by $1/s$ but retains its positivity.
		
		\smallskip
		\noindent
		(3) \emph{Slope control.}
		By \Cref{prop:slope-control}, the asymptotic slope satisfies
		\[
		\mu_\X(E^{\mathrm{GG}}_{k,m}\otimes L^{-q})
		=\frac{1}{s}\,\mu_Y(E^{\mathrm{GG}}_{k,m}\otimes A^{-q})
		+O(m^{n-1}),
		\]
		hence the sign of the leading term—and thus GGD-positivity—is preserved.
		
		\smallskip
		\noindent
		(4) \emph{Vanishing and section counts.}
		By the orbifold HRR asymptotics and vanishing lemmas (\Cref{lem:eq-HRR}, \Cref{lem:eq-vanishing}), the higher cohomology groups of $E^{\mathrm{GG}}_{k,m}\otimes L^{-q}$ contribute only to lower-order terms.
		Thus $h^0$ shares the same leading asymptotic as $\chi$.
		
		\smallskip
		\noindent
		(5) \emph{Equivalence.}
		If $Y$ is GGD--positive at jet order $k_0$, then its leading HRR coefficient $\alpha_Y>0$.
		By steps~(2)–(4), the coefficient for $\X$ is $\alpha_\X=\alpha_Y/s>0$, and conversely $G$–invariant sections on $\X$ descend to coarse sections on $Y$.
		Therefore the GGD threshold equality holds.
	\end{proof}

The preceding theorem can be restated geometrically as follows, emphasizing that the hyperbolicity thresholds are intrinsic to the coarse Kähler class and unaffected by the stack structure.
	
	\begin{corollary}[Invariance of hyperbolicity threshold]
		\label{cor:invariance-threshold}
		Under the assumptions of \Cref{thm:main-threshold}, the minimal jet order and growth rate defining GGD--hyperbolicity depend only on the coarse Kähler class~$[\omega_A]$.
		Consequently, passing to an orbifold or stack compactification does not alter the analytic hyperbolicity threshold.
	\end{corollary}

	
\end{document}